\documentclass[12pt,a4paper]{article}

\usepackage[utf8]{inputenc}
\usepackage[english]{babel}
\usepackage{amsmath}
\usepackage{amsthm}
\usepackage{amsfonts}
\usepackage{amssymb}
\usepackage{graphicx}
\usepackage{e-jc}

\usepackage{latexsym}
\usepackage{url}
\usepackage{subfig}
\usepackage{caption}
\usepackage{multicol}
\usepackage{float}
\usepackage{varioref}
\usepackage{adjustbox}
\usepackage{tikz}
\usepackage{lipsum}

\theoremstyle{plain}
\newtheorem{theorem}{Theorem}
\newtheorem{lemma}[theorem]{Lemma}
\newtheorem{corollary}[theorem]{Corollary}
\newtheorem{proposition}[theorem]{Proposition}

\theoremstyle{definition}
\newtheorem{definition}[theorem]{Definition}

\theoremstyle{remark}

\usepackage[colorlinks=true,citecolor=black,linkcolor=black,urlcolor=blue]{hyperref}

\title{\bf On the Minimum Number of Monochromatic Generalized Schur Triples}
\author{Thotsaporn Thanatipanonda \qquad Elaine Wong\\ \small Science Division\\[-0.8ex]
\small Mahidol University International College\\[-0.8ex] 
\small Nakhon Pathom, Thailand\\
\small\tt \{thotsaporn,wongey\}@gmail.com}
\date{\dateline{Sept. 25, 2016}{Apr 23, 2017}}

\begin{document}

\maketitle

\begin{abstract}
\noindent
The solution to the problem of finding the minimum number of monochromatic triples $(x,y,x+ay)$ with $a\geq 2$ being a fixed positive integer over any 2-coloring of $[1,n]$ was conjectured by Butler, Costello, and Graham (2010) and Thanathipanonda (2009).  We solve this problem using a method based on Datskovsky's proof (2003) on the minimum number of monochromatic Schur triples $(x,y,x+y)$.  We do this by exploiting the combinatorial nature of the original proof and adapting it to the general problem.

  
\end{abstract}

\section{Introduction}

\noindent Ramsey theory has a rich history first popularized in 1935 by Erd\"os and Szekeres in their seminal paper \cite{erdos}.  We investigate a part of the theory that was orginally developed by Issai Schur.  The formulation of Schur's theorem was first derived from Van der Waerden's theorem in 1927 \cite{vanderwaarden}.  Van der Waerden proved that any $r$-coloring of $\mathbb{Z}^+$ must admit a monochromatic 3-term arithmetic progression $\lbrace a, a+d, a+2d \rbrace$ for some $a,d>1$.  A particular choice of $x,y$ and $z$ in terms of $a$ and $d$ admits a monochromatic solution to $x+y=2z$, on a plane whose coordinates are the positive integers.  Hence, a similar question regarding the coloring of monochromatic solutions to a simpler equation can be posed; namely, does there exist a least positive integer $s=s(r)$ such that for any $r$-coloring of $[1,s]$ there is a monochromatic solution to $x + y = z$?  Schur determined that the answer is yes, and we call the solution $(x,y,z)$ to such an equation a Schur triple.
\vspace{0.3cm}

\noindent In 1959, Goodman \cite{goodman} was able to determine the minimum number of monochromatic triangles under a 2-edge coloring of a complete graph on $n$ vertices, which turned out to be the same order as the average, $\frac{n^3}{24}+\mathcal{O}(n^2)$.  This motivated Graham \cite{graham} to pose the problem of finding the minimum number of monochromatic Schur triples over any 2-coloring of $[1,n]$ at a conference and to offer 100 USD for the result.  Graham initially conjectured that the average value should be the minimum at $\frac{n^2}{16}+\mathcal{O}(n)$.  However, Zeilberger and his student Robertson \cite{zeilberger} used discrete calculus to show that the minimum number must be $\frac{n^2}{22}+\mathcal{O}(n)$ and won the cash prize.  Around the same time, Schoen \cite{schoen}, followed four years later by Datskovsky \cite{datmono}, furnished different proofs using Fourier analysis to show that indeed, $\frac{n^2}{22}+\mathcal{O}(n)$ is the correct minimum.   Ultimately, their idea had reduced to one in combinatorics. More recently in 2009, Thanatiponanda \cite{aek} confirmed the result using a new technique with computer algebra and a greedy algorithm.  He conjectured a minimum number of monochromatic Schur triples for all $r$-colorings and a minimum number of monochromatic triples satisfying $x+ay=z$ for a fixed integer $a\geq 2$ over any 2-coloring of $[1,n]$.  We solve the latter part of the conjecture in this paper using a purely combinatorial approach.

\section{The Minimum Number of Monchromatic Schur Triples}

We first show how to find the minimum coloring of $x+y=z$ in an elementary way using the method by Datkovsky \cite{datmono}.  Then, we explore the more general case in the next two sections.  We start by employing a 2-coloring on all integers in $[1,n]$ for $n<\infty$ and count the number of monochromatic Schur triples $(x,y,z)$ where $z=x+y$.  Denote the colors to be red ($R$) and blue ($B$).
\vspace{0.2cm}

\noindent The number of Schur triples includes the number of monochromatic Schur triples $|\mathcal{M}(n)|$ and non-monochromatic Schur triples $|\mathcal {N}(n)|$.  

\begin{lemma}\label{numberofschurtriples}
\begin{equation*}
|\mbox{Schur Triples}|= |\mathcal{M}(n)|+|\mathcal{N}(n)|.
\end{equation*}
\end{lemma}

\begin{lemma}\label{countschurtriples}
The number of Schur triples in $[1,n]$ is $\displaystyle \frac{1}{2}\binom{n}{2}.$
\end{lemma}

\begin{proof}
Observe that a Schur triple can be defined by simply choosing numbers for $x$ and $z$ which gives two triples  $(x,z-x,z)$ and $(z-x,x,z)$.
\end{proof}

\noindent Next we show $|\mathcal{N}(n)|$ can be written in the form of $\mu_B, \mu_R$
and $|N^+|$ all of which are defined as follows:

\begin{definition}
$\mu_B$ denotes the number of blue colorings  of coordinates on $\left[1,n\right]$. 
Similarly, $\mu_R$ denotes the number of red colorings  of coordinates on $\left[1,n\right]$.
\end{definition} 

\noindent Note that $\mu_B+\mu_R=n$. 
\begin{definition}
The set of non-monochromatic pairs in $[1,n]\times[1,n]$ will be denoted as $N(n)$.  In particular, we denote two subsets as follows:
\begin{center}
$N^-=\lbrace (x,y)|\mbox{ } x+y\leq n, x>y \rbrace$\\
$N^+=\lbrace (x,y)|\mbox{ } x+y>n, x<y \rbrace$
\end{center}
\end{definition}

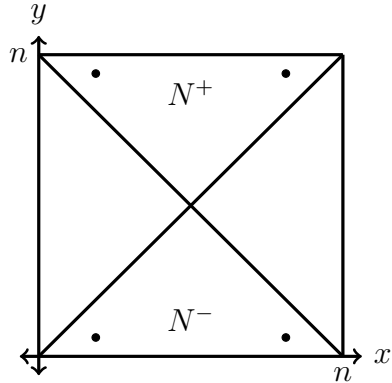
\begin{figure}[ht]
\begin{center}
\begin{tikzpicture}[scale=0.5]
%grid
\def\xmin{-4.5}
\def\xmax{4.5}
\def\ymin{-4.5}
\def\ymax{4.5}
%axes
\draw[very thick, <->] (\xmin,\ymin+0.5)--(\xmax,\ymin+0.5) node[right,scale=1]{$x$};
\draw[very thick, <->] (\xmin+0.5,\ymin)--(\xmin+0.5,\ymax) node[above,scale=1]{$y$};
%label
\draw (\xmax-0.5,\ymin+0.5) node[below,scale=1] {$n$};
\draw (\xmin+0.5,\ymax-0.5) node[left,scale=1] {$n$};
%graph
\draw[very thick] (\xmin+0.5,\ymax-0.5) -- (\xmax-0.5,\ymax-0.5);
\draw[very thick] (\xmin+0.5,\ymin+0.5) -- (\xmax-0.5,\ymax-0.5);
\draw[very thick] (\xmin+0.5,\ymax-0.5) -- (\xmax-0.5,\ymin+0.5);
\draw[very thick] (\xmax-0.5,\ymin+0.5) -- (\xmax-0.5,\ymax-0.5);
\draw (0,\ymax-1.5) node[scale=1]{$N^+$};
\draw (0,\ymin+1.5) node[scale=1]{$N^-$};
%dots
\draw[fill] (\xmin+2,\ymax-1) circle [radius=0.1];
\draw[fill] (\xmax-2,\ymax-1) circle [radius=0.1];
\draw[fill] (\xmin+2,\ymin+1) circle [radius=0.1];
\draw[fill] (\xmax-2,\ymin+1) circle [radius=0.1];
\end{tikzpicture}
\caption{The sets $N^-$ and $N^+$}
\label{fig:graphxplusywithdots}
\end{center}
\end{figure}

\begin{proposition}\label{nonmono1}
$|\mathcal{N}(n)| = \dfrac{1}{2}(2\mu_R\mu_B-|N^+|)+\mathcal{O}(n).$
\end{proposition}

\begin{proof}
\begin{equation*}
\begin{split}
|\mathcal{N}(n)| &= \dfrac{1}{2}|\mbox{non-monochromatic pairs}| \\
&=\dfrac{1}{2}(|N^+|+|N^-|+|N^-|)+\mathcal{O}(n)\\
&=\dfrac{1}{2}(2\mu_R\mu_B-|N^+|)+\mathcal{O}(n).
\end{split}
\end{equation*}

\noindent Each non-monochromatic triple gives two non-monochromatic pairs which gives the first equality. To get the second equality, we observe that the pairs in $N^-$ will contribute to two triples but the pairs in $N^+$ will only contribute to one.  For example, in $[1,10]$, $(5,3)$ gives the triples $(3,2,5)$ and $(5,3,8)$.  But, in $[1,10]$, $(8,9)$ only gives $(1,8,9)$. Finally, the last equality comes from the fact that $|N^+|+|N^-|=\mu_R\mu_B+\mathcal{O}(n).$
\end{proof}

\noindent By putting together Lemmas \ref{numberofschurtriples} and \ref{countschurtriples} and Proposition \ref{nonmono1}, we obtain the next lemma.

\begin{lemma}\label{Dat}
The number of monochromatic Schur triples under a 2-coloring on $[1,n]$ is
\begin{equation*}
|\mathcal{M}(n)|=\frac{n^2}{4}-\frac{1}{2}(2\mu_R\mu_B-|N^+|)+\mathcal{O}(n).
\end{equation*}
\end{lemma}

\noindent In order to find the minimum value of $|\mathcal{M}(n)|$, we must obtain the lower bound of $|N^+|$ in terms of $\mu_B$ and $\mu_R.$  To do this more efficiently, we denote $D:=|N^-|-|N^+|$ and find an upper bound on $D$ instead.  The proof requires the following notation:

\begin{definition}\label{pair}
Let $S$ be the set of pairs of the form  $\{s,n+1-s\}$ where $1 \leq s \leq n/2.$ Denote by $\mu_{CC'}$ the number of sets $S$ with colorings $C$ and $C'$ in this order. $\gamma n$ is the number of non-monochromatic pairs in $S$.
\end{definition}

\noindent Using our new notation, $\gamma n = \mu_{RB}+\mu_{BR}$.

\begin{lemma}\label{Dlemma}
Assume, without loss of generality, that $\mu_B \geq \mu_R$.  Then
\begin{equation*}
D\leq \frac{\mu_B^2}{4}.
\end{equation*}
\end{lemma}

\begin{proof}
\noindent
Assuming $1\leq y<x \leq \frac{n}{2}$, denote the sets $X$ and $Y$ as
\begin{equation*}
\begin{split}
X&=\left\{x,n+1-x\right\}\\
Y&=\left\{y,n+1-y\right\}.
\end{split}
\end{equation*}

\noindent
Ordered pairs in $X\times Y$ when colored \textit{differently} are contained in $N^+\cup N^-$, that is:

\begin{equation*}
\begin{split}
(x,y), (n+1-x,y) \in N^-\\
(x,n+1-y), (n+1-x,n+1-y) \in N^+.
\end{split}
\end{equation*}

\noindent We outline all possible colorings of the $X$ and $Y$ sets that contribute to the value of $D$ in the table in Figure \ref{fig:largechart}. We see that with the exception of the first four cases, the contribution to $D$ is 0.
\vspace{0.2cm}

\noindent We now obtain an upper bound of $D$: 
\begin{align}
D{}&= 2\mu_{RR}\mu_{BR}+2\mu_{BB}\mu_{RB}-2\mu_{RR}\mu_{RB}-2\mu_{BB}\mu_{BR}\notag\\
{\label{proof1:2}}&\leq 2\mu_{RR}\mu_{BR}+2\mu_{BB}\mu_{RB}\\
{\label{proof1:3}}&\leq 2\mu_{BB}(\mu_{BR}+\mu_{RB})\\
{\label{proof1:last}}&=\left(\mu_B-\gamma n\right) \gamma n.
\end{align}

\noindent
The inequality (\ref{proof1:3}) comes from our assumption that $\mu_{BB}\geq\mu_{RR}$.  The last equality comes from $\mu_{BB}= \dfrac{\mu_B-\gamma n}{2}$.
\vspace{0.3cm}

\noindent
Calculus shows that the maximum of (\ref{proof1:last}) occurs when $\gamma = \frac{\mu_B}{2n}$ which simplifies our inequality to
\begin{equation*}
D\leq\left(\mu_B-\frac{\mu_B}{2}\right)\frac{\mu_B}{2}=\frac{\mu_B^2}{4}.
\end{equation*}

\end{proof}

\begin{figure}
\begin{centering}
\adjustbox{max width=\linewidth, keepaspectratio}{
\begin{tabular}{|c|c|c|c|c|p{4cm}|c|}
\hline
Case & $x$ & $n+1-x$ & $y$ & $n+1-y$ & Non-Mono Pairs & $D$ \\ \hline
\multicolumn{7}{|l|}{\textit{Monochromatic X set and Non-monochromatic Y set}}\\
\hline
1 & red & red & blue & red & $(x,y)$&+2\\
&&&&&$(n+1-x,y)$&\\
\hline
2 & blue & blue & red & blue &$(x,y)$&+2\\
&&&&&$(n+1-x,y)$&\\
\hline
3 & red & red & red & blue & $(x,n+1-y)$ & --2\\
&&&&&$(n+1-x,n+1-y)$&\\
\hline
4 & blue & blue & blue & red &$(x,n+1-y)$&--2\\
&&&&&$(n+1-x,n+1-y)$&\\
\hline
\multicolumn{7}{|l|}{\textit{Monochromatic X and Y sets}}\\ \hline
5 & red & red & red & red & none & 0\\ \hline
6 & blue & blue & blue & blue & none & 0\\ \hline
7 & red & red & blue & blue & $(x,y)$& 0\\
&&&&&$(n+1-x,y)$&\\
&&&&&$(x,n+1-y)$&\\
&&&&&$(n+1-x,n+1-y)$&\\ \hline
8 & blue & blue & red & red &$(x,y)$&0\\
&&&&&$(n+1-x,y)$&\\
&&&&&$(x,n+1-y)$&\\
&&&&&$(n+1-x,n+1-y)$&\\ \hline
\multicolumn{7}{|l|}{\textit{Non-monochromatic X set and Monochromatic Y set}}\\ \hline
9 & red & blue & red & red &$(n+1-x,y)$&0\\
&&&&&$(n+1-x,n+1-y)$&\\
\hline
10 & blue & red & blue & blue &$(n+1-x,y)$&0\\
&&&&&$(n+1-x,n+1-y)$&\\
\hline
11 & red & blue & blue & blue &$(x,y)$&0\\
&&&&&$(x,n+1-y)$&\\
\hline
12 & blue & red & red & red &$(x,y)$&0\\
&&&&&$(x,n+1-y)$&\\
\hline
\multicolumn{7}{|l|}{\textit{Non-monochromatic X and Y sets}}\\ \hline
13 & red & blue & blue & red &$(x,y)$&0\\
&&&&&$(n+1-x,n+1-y)$&\\
\hline
14 & blue & red & red & blue &$(x,y)$&0\\
&&&&&$(n+1-x,n+1-y)$&\\
\hline
15 & red & blue & red & blue &$(x,n+1-y)$&0\\
&&&&&$(n+1-x,y)$&\\
\hline
16 & blue & red & blue & red &$(x,n+1-y)$&0\\
&&&&&$(n+1-x,y)$&\\
\hline
\end{tabular}}
\caption{Colorings of the elements in sets $X$ and $Y$}
\label{fig:largechart}
\end{centering}
\end{figure}

\newpage
\begin{theorem}\label{theorem1}
Over all 2-colorings of $[1,n]$, the minimum number of monochromatic Schur triples is $\frac{n^2}{22}+\mathcal{O}(n)$.
\end{theorem}

\begin{proof}

\noindent An upper bound of the minimum can be obtained from a coloring on $[1,n]$.  We color $\left[R^{4n/11},B^{6n/11},R^{n/11}\right]$ as illustrated in Figure \ref{fig:optimalcoloring}. This proportion comes from a brute force computer search first proposed by Zeilberger \cite{zeilberger}.
\vspace{0.5cm}

\begin{figure}[ht]
\begin{center}
\begin{tikzpicture}[scale=0.9]
%grid
\def\xmin{-0.5}
\def\xmax{11.5}
\def\ymin{-1}
\def\ymax{1}
%axes
\draw[very thick, <-] (\xmin,\ymin+0.5)--(\xmin+0.5,\ymin+0.5);
\draw[very thick, red, -] (\xmin+0.5,\ymin+0.5)--(\xmin+4.5,\ymin+0.5);
\draw[very thick, blue, -] (\xmin+4.5,\ymin+0.5)--(\xmin+9.5,\ymin+0.5);
\draw[very thick, red, -] (\xmin+9.5,\ymin+0.5)--(\xmin+11,\ymin+0.5);
\draw[very thick, ->] (\xmin+11,\ymin+0.5)--(\xmin+11.5,\ymin+0.5);
%label
\draw (\xmin+0.5,\ymin+0.4) -- (\xmin+0.5,\ymin+0.6);
\draw (\xmin+2.2,\ymin+0.5) node[below,scale=1] {$\frac{4n}{11}$};
\draw[red] (\xmin+2.2,\ymin+0.5) node[above,scale=1] {red};
\draw (\xmin+4.5,\ymin+0.4) -- (\xmin+4.5,\ymin+0.6);
\draw (\xmin+7.2,\ymin+0.5) node[below,scale=1] {$\frac{6n}{11}$};
\draw[blue] (\xmin+7.2,\ymin+0.5) node[above,scale=1] {blue};
\draw (\xmin+9.5,\ymin+0.4) -- (\xmin+9.5,\ymin+0.6);
\draw (\xmin+10.2,\ymin+0.5) node[below,scale=1] {$\frac{n}{11}$};
\draw[red] (\xmin+10.2,\ymin+0.5) node[above,scale=1] {red};
\draw (\xmin+11,\ymin+0.4) -- (\xmin+11,\ymin+0.6);
\end{tikzpicture}
\caption{The Optimal Coloring for $x+y=z$}
\label{fig:optimalcoloring}
\end{center}
\end{figure}
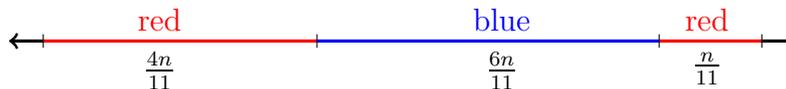

\noindent This coloring will give us $\frac{n^2}{22}+\mathcal{O}(n)$ monochromatic triples.
\vspace{0.2cm}

\noindent Next, we look for the lower bound of the minimum.  By using Lemma \ref{Dlemma} and the fact that $|N^-|+|N^+|=\mu_R\mu_B+\mathcal{O}(n)$, we get:

\begin{equation*}
|N^+| \geq \frac{1}{2}\mu_R\mu_B-\frac{\mu^2_B}{8}+\mathcal{O}(n),
\end{equation*}

\noindent which, together with Lemma \ref{Dat}, gives

\begin{equation}\label{eq:lowerboundM}
|\mathcal{M}(n)|\geq \frac{n^2}{4}-\frac{3}{4}\mu_R\mu_B-\frac{\mu_B^2}{16}+\mathcal{O}(n).
\end{equation}

\noindent The right hand side of (\ref{eq:lowerboundM}) achieves a maximum when $\mu_R=\frac{5n}{11}$ and $\mu_B=\frac{6n}{11}$.  As a result, we get the lower bound for the minimum to be:

\begin{equation*}
|\mathcal{M}(n)|\geq \frac{n^2}{22}+\mathcal{O}(n).
\end{equation*}

\noindent Because the lower and upper bounds match, we have therefore shown the desired result.

\end{proof}

\noindent \textit{Remark.} We can be confident that the bounds for equations
(\ref{proof1:2}) and (\ref{proof1:3}) is sharp relative to the optimal coloring because we know that cases 3 and 4 from Figure \ref{fig:largechart} will not occur and that $\mu_{BR}=0$.
\vspace{0.2cm}

\noindent This method of Datkovsky's also gives the optimal coloring for fixed ratios of $\mu_B$ and $\mu_R.$
\begin{corollary}\label{corollary1}
For any fixed $\mu_B \geq \mu_R$, the coloring on $[1,n]$ that gives the minimum number of monochromatic Schur triples is 
$\left[R^{\frac{n}{2}-\frac{\mu_B}{4}},B^{\mu_B},R^{\frac{n}{2}-\frac{3\mu_B}{4}}\right]$  
for $ \mu_B\leq\frac{2n}{3}$ and 
 $\left[R^{\mu_R},B^{\mu_B}\right]$ for $\mu_B >\frac{2n}{3}$.
\end{corollary}

\begin{proof}
We follow the proof of Lemma \ref{Dlemma} and use
\[ D \leq (\mu_B-\gamma n)\gamma n. \]
For the case $ \mu_B\leq\frac{2n}{3}$, the maximum of $D$ occurs when $\gamma =\frac{\mu_B}{2n}.$ So,
\[|\mathcal{M}(n)|\geq \frac{n^2}{4}-\frac{3}{4}\mu_R\mu_B-\frac{\mu_B^2}{16}+\mathcal{O}(n).\]

\noindent For the case $\mu_B >\frac{2n}{3},$ the maximum of $D$ occurs when
$\gamma = \frac{\mu_R}{n}.$ With similar calculations,
\begin{equation*}
|\mathcal{M}(n)| \geq \frac{n^2}{4}-\mu_R\mu_B+\frac{\mu_R^2}{4}+\mathcal{O}(n).
\end{equation*}
The colorings mentioned in the statement of the corollary give us upper bounds for the minimum, which happens to match the lower bounds.
\end{proof}

\begin{corollary}\label{corollary2}
For any fixed $\mu_B\geq\mu_R$, the coloring on $[1,n]$ that gives the maximum number of monochromatic Schur triples is $\left[R^{\frac{n}{2}-\frac{3\mu_B}{4}},B^{\mu_B},R^{\frac{n}{2}-\frac{\mu_B}{4}}\right]$   for $\mu_B\leq\frac{2n}{3}$ and $\left[B^{\mu_B}, R^{\mu_R}\right]$ for $\mu_B >\frac{2n}{3}$.
\end{corollary}

\begin{proof}
Using a similar calculation to Lemma \ref{Dlemma}, we have that
\[ D \geq -(\mu_B-\gamma n)\gamma n. \]
For the case $ \mu_B\leq\frac{2n}{3}$, the minimum of $D$ occurs when $\gamma =\frac{\mu_B}{2n}.$  So,
\begin{equation*}
|\mathcal{M}(n)| \leq \frac{n^2}{4}-\dfrac{3}{4}\mu_R\mu_B+\dfrac{\mu_B^2}{16}+\mathcal{O}(n).
\end{equation*}
For the case  $\mu_B >\frac{2n}{3}$, the minimum of $D$ occurs when $\gamma =\frac{\mu_R}{n}.$  Therefore, 
\begin{equation*}
|\mathcal{M}(n)| \leq \frac{n^2}{4}-\dfrac{\mu_R\mu_B}{2}-\dfrac{\mu_R^2}{4}+\mathcal{O}(n).
\end{equation*}
The colorings mentioned in the statement of the corollary give us lower bounds of the maximum, which happens to match the upper bounds.
\end{proof}

%%%%%%%%%%%%%%%%%%%%%%%%%%%%%%%%%%%%%%%%%%%%%%%
\section{The Minimum Number of Monochromatic Triples
\texorpdfstring{\\$(x,y,x+2y)$}{}} \label{a=2}

\noindent
The technique illustrated in the previous section can be extended to $x+ay=z$ for any fixed integer $a\geq 2$.  However, the nice symmetry we had previously with the equation $x+y=z$ is no longer there.   In this section we deal with the specific case $a=2$. The general case will be outlined in Section 4.
\vspace{0.2cm}

\noindent We parallel the same argument as in Section 2.  First, we write the number of non-monochromatic triples $|\mathcal{N}(n)|$ in terms of variables we can optimize.

\begin{definition}
Denote by $\mu_{B_1}$ and $\mu_{R_1}$ the number of blue and red colorings respectively on $\left[1,\frac{n}{2}\right]$.  Furthermore, denote by $\mu_{B_2}$ and $\mu_{R_2}$ the number of blue and red colorings respectively on $\left(\frac{n}{2},n\right]$.
\end{definition}

\noindent Note that $\mu_{B_1}+\mu_{R_1}=\frac{n}{2}$ and $\mu_{R_1}+\mu_{R_2}=\mu_R$.

\begin{definition}
The sets of non-monochromatic pairs in $[1,n]\times[1,\frac{n}{2}]$ 
are defined as follows:
\begin{equation*}
\begin{split}
N_x^-=\left\{(x,y)|\, x+2y\leq n,x>2y\right\}\\
N_x^+=\left\{(x,y)| \, x+2y>n,x>2y\right\}\\
N_y^-=\left\{(x,y)| \, x+2y\leq n,x<2y\right\}\\
N_y^+=\left\{(x,y)| \, x+2y>n,x<2y\right\}
\end{split}
\end{equation*}

\end{definition}

\begin{figure}[ht]
\begin{center}
\begin{tikzpicture}[scale=0.65]
%grid
\def\xmin{-4.5}
\def\xmax{4.5}
\def\ymin{-4.5}
\def\ymax{4.5}
%axes
\draw[very thick, <->] (\xmin,\ymin+0.5)--(5.5,\ymin+0.5) node[right,scale=1]{$x$};
\draw[very thick, <->] (\xmin+0.5,\ymin)--(\xmin+0.5, 1.5) node[above,scale=1]{$y$};
%label
\draw (\xmax-0.5,\ymin+0.5) node[below,scale=1] {$n$};
\draw (\xmin+0.5,0) node[left,scale=1] {$\frac{n}{2}$};
%graph
\draw[very thick] (\xmin+0.5,0) -- (\xmax-0.5,0);
\draw[very thick] (\xmin+0.5,\ymin+0.5) -- (\xmax-0.5,0);
\draw[very thick] (\xmin+0.5,0) -- (\xmax-0.5,\ymin+0.5);
\draw[very thick] (\xmax-0.5,\ymin+0.5) -- (\xmax-0.5,0);
\draw (0,-1) node[scale=1]{$N_y^+$};
\draw (0,\ymin+1.3) node[scale=1]{$N_x^-$};
\draw (\xmin+2,\ymin+2.5) node[scale=1]{$N_y^-$};
\draw (\xmax-2,\ymin+2.5) node[scale=1]{$N_x^+$};
%dots
\draw[fill] (\xmin+2.5,\ymin+4) circle [radius=0.1];
\draw[fill] (\xmax-2.5,\ymin+4) circle [radius=0.1];
\draw[fill] (\xmin+2.5,\ymin+1) circle [radius=0.1];
\draw[fill] (\xmax-2.5,\ymin+1) circle [radius=0.1];
\end{tikzpicture}
\caption{The sets $N_x^-$, $N_x^+$ $N_y^-$ and $N_y^+$.}
\label{fig:graphxplus2ywithdots}
\end{center}
\end{figure}
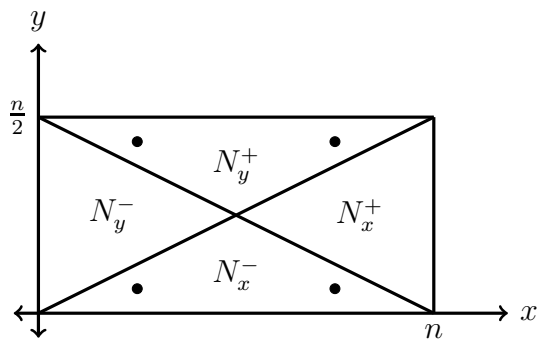

\begin{proposition}\label{nonmono2}
$|\mathcal{N}(n)| \leq \dfrac{1}{2}\left( \dfrac{\mu_R\mu_B}{2}+\mu_R\mu_{B_1}+\mu_B\mu_{R_1}
+|N_x^-|-|N_y^+| \right)+\mathcal{O}(n).$
\end{proposition}

\begin{proof}
We count sets of ordered pairs that come from non-monochromatic triples of type $(x,y,z)$ where $z=x+2y$ as follows:
\begin{equation*}
\begin{split}
\nu_1&= |\{\mbox{non-monochromatic }(x,y) \mbox{ pairs}\}|=|\left\{(x,y)| \, x+2y\leq n\right\}|\\
\nu_2&= |\{\mbox{non-monochromatic }(y,z) \mbox{ pairs}\}|=|\left\{(x,y)| \, y > 2x\right\}|
=|\left\{(x,y)| \, x > 2y \right\}|\\
\nu_3&= |\{\mbox{non-monochromatic }(x,z) \mbox{ pairs}\}|
=|\left\{(x,y)| \, 2 \mbox{ divides } (y-x),x<y\right\}|.
\end{split}
\end{equation*}

\noindent We then have that:
\begin{equation*}
\begin{split}
|\mathcal{N}(n)| &=\dfrac{1}{2}(\nu_1+\nu_2+\nu_3)\\
&\leq \dfrac{1}{2}\left( \dfrac{\mu_R\mu_B}{2}+ 2|N_x^-|+|N_x^+|+|N_y^-|  \right)+\mathcal{O}(n) \\
&= \dfrac{1}{2}\left( \dfrac{\mu_R\mu_B}{2}+\mu_R\mu_{B_1}+\mu_B\mu_{R_1}
+|N_x^-|-|N_y^+| \right)+\mathcal{O}(n).
\end{split}
\end{equation*}

\noindent To obtain the second inequality, observe that
$\nu_1=|N_x^-|+|N_y^-|+\mathcal{O}(n)$, 
$\nu_2=|N_x^-|+|N_x^+|+\mathcal{O}(n)$
and $\nu_3 \leq \dfrac{\mu_R\mu_B}{2}$ (refer to Lemma 3 of Thanatipanonda \cite{aek}).
The last equality comes from $|N_x^-|+|N_x^+|+|N_y^-|+|N_y^+|=\mu_R\mu_{B_1}+\mu_B\mu_{R_1}$
which can be seen from Figure \ref{fig:graphxplus2ywithdots}.
\end{proof}

\noindent The next lemma follows immediately from Proposition \ref{nonmono2}.

\begin{lemma}\label{M2}
The number of monochromatic triples of type $(x,y,x+2y)$ under a 2-coloring on $[1,n]$ is
\begin{equation*}
|\mathcal{M}(n)| \geq \frac{n^2}{4}- \dfrac{1}{2}\left( \dfrac{\mu_R\mu_B}{2}+\mu_R\mu_{B_1}
+\mu_B\mu_{R_1}+|N_x^-|-|N_y^+| \right)+\mathcal{O}(n).
\end{equation*}
\end{lemma}

\noindent Minimizing $|\mathcal{M}(n)|$ can be reduced to finding the upper bound of $D_2 :=|N_x^-|-|N_y^+|$ in terms of $\mu_R, \mu_B, \mu_{R_1}, \mu_{B_1}$.  $D_2$ is much more difficult to count than in the Schur triple case.

\begin{definition}\label{halfpair}
Let $S$ be the set of pairs of the form  $\{s,\frac{n}{2}+1-s\}$ where $1 \leq s \leq n/4.$ Denote by $\mu_{CC'}^{(1)}$ the number of sets $S$ with colorings $C$ and $C'$ in this order.  The superscript (1) refers to the coloring of pairs on $\left[1,\frac{n}{2}\right]$. Also, denote by $\gamma_1 n$ the number of non-monochromatic pairs in $S$.
\end{definition}

\noindent With this notation, $\gamma_1 n= \mu_{RB}^{(1)}+\mu_{BR}^{(1)}.$

\begin{definition}
Define sets $X$ and $Y_1$ as follows,

\begin{equation*}
\begin{split}
X&=\left\{x,n+1-x\right\},  \;\  1 \leq x \leq \frac{n}{2}\\
Y_1&=\left\{y,\frac{n}{2}+1-y\right\}, \;\ 1 \leq y \leq \frac{n}{4}.
\end{split}
\end{equation*}

\noindent The direct product $\mu_{CC'} \otimes \mu_{EE'}^{(1)}$ is defined by the number of pairs $(X,Y_1)$ where $X$ has the coloring $\{C,C'\}$ and $Y_1$ has the coloring $\{E,E'\}$ under the condition $2y < x$.

\end{definition}

\begin{lemma}\label{D2lemma}
\begin{equation*}
D_2 = 2\mu_{RR}\otimes\mu_{BR}^{(1)}+2\mu_{BB}\otimes\mu_{RB}^{(1)}
-2\mu_{RR}\otimes\mu_{RB}^{(1)}-2\mu_{BB}\otimes\mu_{BR}^{(1)}.
\end{equation*}
\end{lemma}

\begin{proof}
\noindent Assuming $1\leq 2y<x\leq \frac{n}{2}$, we observe that the ordered pairs in $X \times Y_1$ when colored \textit{differently} are contained in $N_x^-\cup N_y^+$, that is:
\begin{equation*}
\begin{split}
(x,y), (n+1-x,y) \in N_x^-\\
\left(x,\frac{n}{2}+1-y\right), \left(n+1-x,\frac{n}{2}+1-y\right) \in N_y^+.
\end{split}
\end{equation*}

\noindent The table in Figure \ref{fig:smallchart} shows that there are only four cases that contribute 
any value to $D_2$, while the other cases contribute $0$, similar to the table in Figure \ref{fig:largechart}.
\vspace{0.2cm}

\begin{figure}[ht]
\begin{centering}
\begin{tabular}{|c|c|c|c|c|p{4cm}|c|}
\hline
\multicolumn{7}{|l|}{\textit{Monochromatic $X$ set and Non-monochromatic $Y_1$ set}}\\
\hline
Case & $x$ & $n+1-x$ & $y$ & $\frac{n}{2}+1-y$ & Non-Mono Pairs & $D_2$ \\
\hline
1 & red & red & blue & red & $(x,y)$&+2\\
&&&&&$(n+1-x,y)$&\\
\hline
2 & blue & blue & red & blue &$(x,y)$&+2\\
&&&&&$(n+1-x,y)$&\\
\hline
3 & red & red & red & blue & $\left(x,\frac{n}{2}+1-y\right)$ & --2\\
&&&&&$\left(n+1-x,\frac{n}{2}+1-y\right)$&\\
\hline
4 & blue & blue & blue & red &$\left(x,\frac{n}{2}+1-y\right)$&--2\\
&&&&&$\left(n+1-x,\frac{n}{2}+1-y\right)$&\\
\hline
\end{tabular}
\caption{Colorings of the elements in sets $X$ and $Y_1$ where $D_2\neq 0$}
\label{fig:smallchart}
\end{centering}
\end{figure}

\noindent The result follows immediately.
\end{proof}

\noindent The next proposition gives the upper bound for $D_2$.

\begin{proposition}\label{d2bound}
Assume, without loss of generality, that $\mu_B \geq \mu_R$ and suppose the number of non-monochromatic pairs in S, $\gamma_1 n$, is fixed. Then
\[D_2 \leq   2 A_1,\]
where $A_1$ is the largest possible area under the curve in Figure \ref{fig:d2graphcombined}, with a base of length $\gamma_1 n$ for $\gamma_1 \leq \frac{1}{4}$.
\end{proposition}

\begin{figure}[ht]
\begin{center}
\begin{tikzpicture}[scale=0.7]
%grid
\def\xmin{-5.5}
\def\xmax{5.5}
\def\ymin{-4.5}
\def\ymax{4.5}
%axes
\draw[very thick, <->] (\xmin,\ymin+2.5)--(\xmax,\ymin+2.5) node[right,scale=1]{$y$ position};
\draw[very thick, <->] (\xmin+0.5,\ymin+2)--(\xmin+0.5,\ymax) node[above,scale=1]{pairs gained};
%label
\draw (0,\ymin+2.5) node[below,scale=1] {$\frac{n}{4}$};
\draw (0,\ymin+2.4) -- (0,\ymin+2.6);
\draw (\xmax-0.5,\ymin+2.5) node[below,scale=1] {$\frac{n}{2}$};
\draw (\xmax-0.5,\ymin+2.4) -- (\xmax-0.5,\ymin+2.6);
\draw (\xmin+0.5,\ymax-1.5) node[left,scale=1] {$\mu_{BB}$};
\draw (\xmin+0.4,\ymax-1.5) -- (\xmin+0.6,\ymax-1.5);
\draw (\xmin+0.5,\ymax-2.9) node[left,scale=1] {$\mu_{RR}$};
\draw (\xmin+0.4,\ymax-2.9) -- (\xmin+0.6,\ymax-2.9);
\draw (\xmin+0.5,\ymax-5) node[left,scale=1] {$\mu_{BB}-\mu_{RR}$};
\draw (\xmin+0.4,\ymax-5) -- (\xmin+0.6,\ymax-5);
\draw (\xmin+1.5,\ymin+2.5) node[below,scale=0.7] {$\frac{\mu_{RR}}{2}$};
\draw (\xmin+1.5,\ymin+2.4) -- (\xmin+1.5,\ymin+2.6);
\draw (\xmin+4,\ymin+2.5) node[below,scale=0.7] {$\frac{n}{4}-\frac{\mu_{BB}}{2}$};
\draw (\xmin+4,\ymin+2.4) -- (\xmin+4,\ymin+2.6);
%shading
\draw[fill=lightgray] (\xmin+1.7,\ymin+2.5) rectangle (\xmin+3.8,\ymax-1.5);
\draw [|-|, thick] (\xmin+1.7,\ymin+2.7) -- (\xmin+3.8,\ymin+2.7);
\draw (\xmin+2.75,\ymin+2.7) node[above,scale=0.7] {$\gamma_1 n$};
\draw (\xmin+2.75,\ymin+5) node[scale=2]{$A_1$};
%graph
\draw[very thick] (\xmin+0.5,\ymax-5) -- (\xmin+1.5,\ymax-1.5) -- (\xmin+4,\ymax-1.5) -- (0,\ymin+2.5);
\draw[very thick, dashed] (\xmin+1.2,\ymin+2.5) -- (\xmin+2.2,\ymax-2.9) -- (\xmin+4.425,\ymax-2.9) -- (0,\ymin+2.5);
\end{tikzpicture}
\caption{The upper bound of $D_2$}
\label{fig:d2graphcombined}
\end{center}
\end{figure}
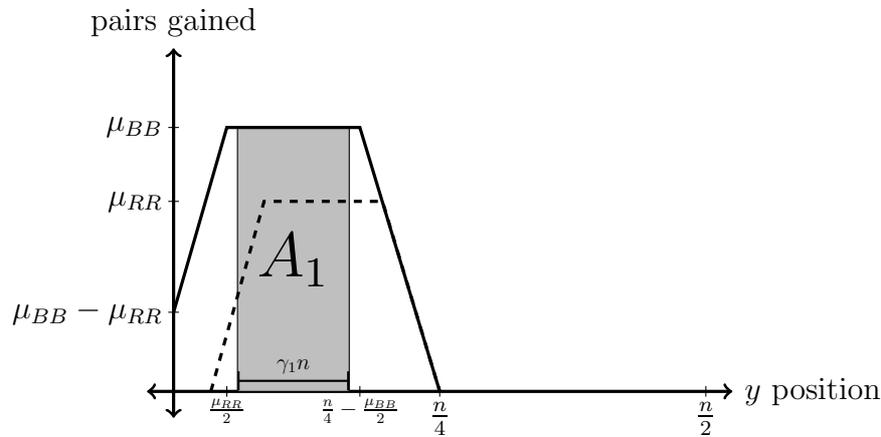

\begin{proof}
\noindent From Lemma \ref{D2lemma},
\begin{align*}
D_2 &= 2\mu_{RR}\otimes\mu_{BR}^{(1)}+2\mu_{BB}\otimes\mu_{RB}^{(1)}
-2\mu_{RR}\otimes\mu_{RB}^{(1)}-2\mu_{BB}\otimes\mu_{BR}^{(1)}\notag\\
{}&= 2(\mu_{BB}\otimes\mu_{RB}^{(1)}-\mu_{RR}\otimes\mu_{RB}^{(1)})
+2(\mu_{RR}\otimes\mu_{BR}^{(1)}-\mu_{BB}\otimes\mu_{BR}^{(1)}).
\end{align*}
\noindent We configure $X$ to gain the maximum of $\mu_{BB}\otimes\mu_{RB}^{(1)}-\mu_{RR}\otimes\mu_{RB}^{(1)}$ by coloring the far left of the interval $\left[1,\frac{n}{2}\right]$ red and the remainder of the interval blue, which is justified by the condition $2y<x$ . With this set up for $X$, we can count the number of pairs gained for every $y$ in $Y_1$ as shown in Figure \ref{fig:A1}.
\vspace{0.2cm}

\noindent Similarly, the configuration of X to gain the maximum of
$\mu_{RR}\otimes\mu_{BR}^{(1)}-\mu_{BB}\otimes\mu_{BR}^{(1)}$
is when we color the far left of the interval $\left[1,\frac{n}{2}\right]$ blue and the remainder of the interval red. With this set up for $X$, we can count the number of pairs gained for every $y$ in $Y_1$ as shown in Figure \ref{fig:A2}.
\vspace{0.2cm}

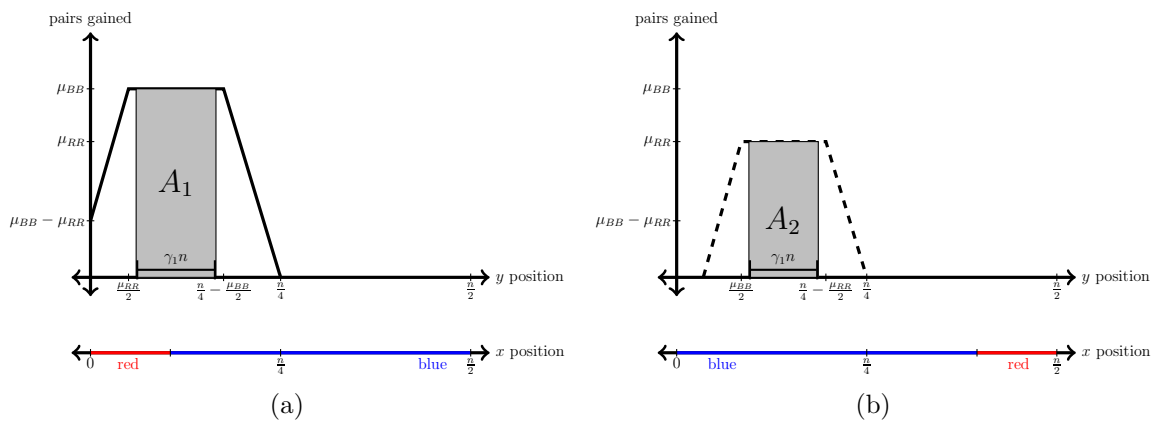
\begin{figure}[ht]

\subfloat[\label{fig:A1}]{
\begin{tikzpicture}[scale=0.5]
%grid
\def\xmin{-5.5}
\def\xmax{5.5}
\def\ymin{-4.5}
\def\ymax{4.5}
\def\scale{0.5}
%axes
\draw[very thick, <->] (\xmin,\ymin+2.5)--(\xmax,\ymin+2.5) node[right,scale=\scale]{$y$ position};
\draw[very thick, <->] (\xmin,\ymin+0.5)--(\xmax,\ymin+0.5) node[right,scale=\scale]{$x$ position};
\draw[very thick, red] (\xmin+0.5,\ymin+0.5)--(\xmin+2.6,\ymin+0.5);
\draw[very thick, blue] (\xmin+2.6,\ymin+0.5)--(\xmax-0.5,\ymin+0.5);
\draw[very thick, <->] (\xmin+0.5,\ymin+2)--(\xmin+0.5,\ymax) node[above,scale=\scale]{pairs gained};
%label
\draw (\xmin+1.5,\ymin+2.5) node[below,scale=\scale] {$\frac{\mu_{RR}}{2}$};
\draw (\xmin+4,\ymin+2.5) node[below,scale=\scale] {$\frac{n}{4}-\frac{\mu_{BB}}{2}$};
\draw (0,\ymin+2.5) node[below,scale=\scale] {$\frac{n}{4}$};
\draw (\xmax-0.5,\ymin+2.5) node[below,scale=\scale] {$\frac{n}{2}$};
\draw (\xmin+1.5,\ymin+2.4) -- (\xmin+1.5,\ymin+2.6);
\draw (\xmin+4,\ymin+2.4) -- (\xmin+4,\ymin+2.6);
\draw (0,\ymin+2.4) -- (0,\ymin+2.6);
\draw (\xmax-0.5,\ymin+2.4) -- (\xmax-0.5,\ymin+2.6);
\draw (\xmin+0.5,\ymin+0.5) node[below,scale=\scale] {$0$};
\draw (\xmin+1.5,\ymin+0.5) node[below,red,scale=\scale] {red};
\draw (0,\ymin+0.5) node[below,scale=\scale] {$\frac{n}{4}$};
\draw (\xmax-1.5,\ymin+0.5) node[below,blue,scale=\scale] {blue};
\draw (\xmax-0.5,\ymin+0.5) node[below,scale=\scale] {$\frac{n}{2}$};
\draw (\xmin+0.5,\ymin+0.4) -- (\xmin+0.5,\ymin+0.6);
\draw (\xmin+2.6,\ymin+0.4) -- (\xmin+2.6,\ymin+0.6);
\draw (0,\ymin+0.4) -- (0,\ymin+0.6);
\draw (\xmax-0.5,\ymin+0.4) -- (\xmax-0.5,\ymin+0.6);
\draw (\xmin+0.5,\ymax-1.5) node[left,scale=\scale] {$\mu_{BB}$};
\draw (\xmin+0.4,\ymax-1.5) -- (\xmin+0.6,\ymax-1.5);
\draw (\xmin+0.5,\ymax-2.9) node[left,scale=\scale] {$\mu_{RR}$};
\draw (\xmin+0.4,\ymax-2.9) -- (\xmin+0.6,\ymax-2.9);
\draw (\xmin+0.5,\ymax-5) node[left,scale=\scale] {$\mu_{BB}-\mu_{RR}$};
\draw (\xmin+0.4,\ymax-5) -- (\xmin+0.6,\ymax-5);
%graph
\draw[very thick] (\xmin+0.5,\ymax-5) -- (\xmin+1.5,\ymax-1.5) -- (\xmin+4,\ymax-1.5) -- (0,\ymin+2.5);
%shading
\draw[fill=lightgray] (\xmin+1.7,\ymin+2.5) rectangle (\xmin+3.8,\ymax-1.5);
\draw [|-|, thick] (\xmin+1.7,\ymin+2.7) -- (\xmin+3.8,\ymin+2.7);
\draw (\xmin+2.75,\ymin+2.7) node[above,scale=\scale] {$\gamma_1 n$};
\draw (\xmin+2.75,\ymin+5) node[scale=\scale+\scale]{$A_1$};
\end{tikzpicture}
}%
\subfloat[\label{fig:A2}]{
\begin{tikzpicture}[scale=0.5]
%grid
\def\xmin{-5.5}
\def\xmax{5.5}
\def\ymin{-4.5}
\def\ymax{4.5}
\def\scale{0.5}
%axes
\draw[very thick, <->] (\xmin,\ymin+2.5)--(\xmax,\ymin+2.5) node[right,scale=\scale]{$y$ position};
\draw[very thick, <->] (\xmin,\ymin+0.5)--(\xmax,\ymin+0.5) node[right,scale=\scale]{$x$ position};
\draw[very thick, blue] (\xmin+0.5,\ymin+0.5)--(\xmax-2.6,\ymin+0.5);
\draw[very thick, red] (\xmax-2.6,\ymin+0.5)--(\xmax-0.5,\ymin+0.5);
\draw[very thick, <->] (\xmin+0.5,\ymin+2)--(\xmin+0.5,\ymax) node[above,scale=\scale]{pairs gained};
%label
\draw (\xmin+2.2,\ymin+2.5) node[below,scale=\scale] {$\frac{\mu_{BB}}{2}$};
\draw (\xmin+2.2,\ymin+2.4) -- (\xmin+2.2,\ymin+2.6);
\draw (\xmin+4.425,\ymin+2.5) node[below,scale=\scale] {$\frac{n}{4}-\frac{\mu_{RR}}{2}$};
\draw (\xmin+4.425,\ymin+2.4) -- (\xmin+4.425,\ymin+2.6);
\draw (0,\ymin+2.5) node[below,scale=\scale] {$\frac{n}{4}$};
\draw (0,\ymin+2.4) -- (0,\ymin+2.6);
\draw (\xmax-0.5,\ymin+2.5) node[below,scale=\scale] {$\frac{n}{2}$};
\draw (\xmax-0.5,\ymin+2.4) -- (\xmax-0.5,\ymin+2.6);
\draw (\xmin+0.5,\ymin+0.5) node[below,scale=\scale] {$0$};
\draw (\xmin+1.7,\ymin+0.5) node[below,blue,scale=\scale] {blue};
\draw (0,\ymin+0.5) node[below,scale=\scale] {$\frac{n}{4}$};
\draw (\xmax-2.6,\ymin+0.4) -- (\xmax-2.6,\ymin+0.6);
\draw (\xmax-1.5,\ymin+0.5) node[below,red,scale=\scale] {red};
\draw (\xmax-0.5,\ymin+0.5) node[below,scale=\scale] {$\frac{n}{2}$};
\draw (\xmin+0.5,\ymin+0.4) -- (\xmin+0.5,\ymin+0.6);
\draw (0,\ymin+0.4) -- (0,\ymin+0.6);
\draw (\xmax-0.5,\ymin+0.4) -- (\xmax-0.5,\ymin+0.6);
\draw (\xmin+0.5,\ymax-1.5) node[left,scale=\scale] {$\mu_{BB}$};
\draw (\xmin+0.4,\ymax-1.5) -- (\xmin+0.6,\ymax-1.5);
\draw (\xmin+0.5,\ymax-2.9) node[left,scale=\scale] {$\mu_{RR}$};
\draw (\xmin+0.4,\ymax-2.9) -- (\xmin+0.6,\ymax-2.9);
\draw (\xmin+0.5,\ymax-5) node[left,scale=\scale] {$\mu_{BB}-\mu_{RR}$};
\draw (\xmin+0.4,\ymax-5) -- (\xmin+0.6,\ymax-5);
%graph
\draw[very thick, dashed] (\xmin+1.2,\ymin+2.5) -- (\xmin+2.2,\ymax-2.9) -- (\xmin+4.425,\ymax-2.9) -- (0,\ymin+2.5);
%shading
\draw[fill=lightgray] (\xmin+2.4,\ymin+2.5) rectangle (\xmin+4.225,\ymax-2.9);
\draw [|-|, thick] (\xmin+2.4,\ymin+2.7) -- (\xmin+4.225,\ymin+2.7);
\draw (\xmin+3.3,\ymin+2.7) node[above,scale=\scale] {$\gamma_1 n$};
\draw (\xmin+3.3,\ymin+4) node[scale=\scale+\scale]{$A_2$};
\end{tikzpicture}
}%

\caption{Counting $D_2$ with different colorings of $X$.}

\end{figure}

\noindent Finally, since $\mu_{RB}^{(1)}+\mu_{BR}^{(1)}=\gamma_1 n,$
and the number of pairs gained shown in Figure \ref{fig:A2} is clearly dominated by the number of pairs gained in Figure \ref{fig:A1}, the result follows.
\end{proof}

\noindent We now find the upper bound for $|\mathcal{N}(n)|.$

\begin{proposition}\label{N2}
Over all 2-colorings of $[1,n]$, the maximum number of non-monochro-matic triples satisfying $x+2y=z$ is
\begin{equation*} 
|\mathcal{N}(n)| \leq \frac{5n^2}{22}+\mathcal{O}(n).\\
\end{equation*}
\end{proposition}

\begin{proof}

Assume, without loss of generality, that $\mu_B\geq\mu_R$.  The proof ultimately depends on determining the area under the curve of Figure \ref{fig:d2graphcombined}, which we break down into three cases according to the value of $\gamma_1$. The three cases are illustrated in Figure \ref{fig:3cases}.

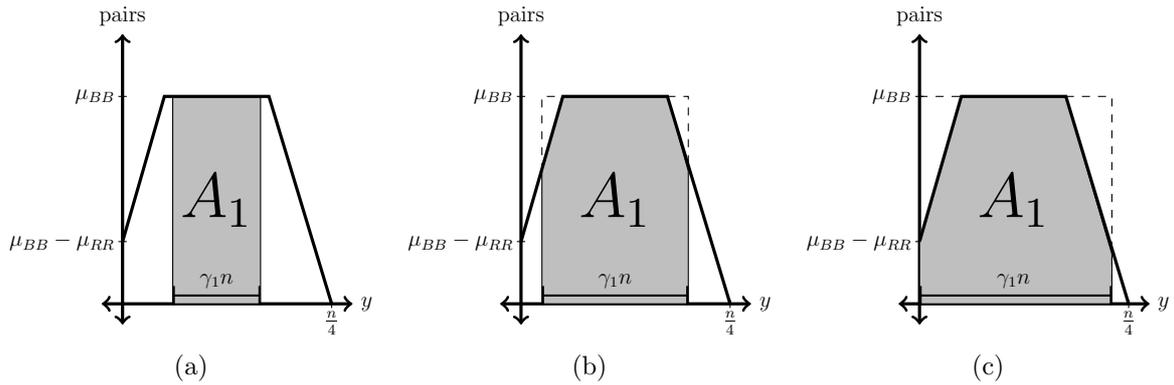
\begin{figure}[ht]
\begin{center}
\subfloat[\label{fig:case1area}]{
\begin{tikzpicture}[scale=0.55]
%grid
\def\xmin{-5.5}
\def\xmax{5.5}
\def\ymin{-4.5}
\def\ymax{4.5}
\def\scale{0.7}
%axes
\draw[very thick, <->] (\xmin,\ymin+2.5)--(0.5,\ymin+2.5) node[right,scale=\scale]{$y$};
\draw[very thick, <->] (\xmin+0.5,\ymin+2)--(\xmin+0.5,\ymax) node[above,scale=\scale]{pairs};
%label
\draw (0,\ymin+2.5) node[below,scale=\scale] {$\frac{n}{4}$};
\draw (0,\ymin+2.4) -- (0,\ymin+2.6);
\draw (\xmin+0.5,\ymax-1.5) node[left,scale=\scale] {$\mu_{BB}$};
\draw (\xmin+0.4,\ymax-1.5) -- (\xmin+0.6,\ymax-1.5);
\draw (\xmin+0.5,\ymax-5) node[left,scale=\scale] {$\mu_{BB}-\mu_{RR}$};
\draw (\xmin+0.4,\ymax-5) -- (\xmin+0.6,\ymax-5);
%shading
\draw[fill=lightgray] (\xmin+1.7,\ymin+2.5) rectangle (\xmin+3.8,\ymax-1.5);
\draw [|-|, thick] (\xmin+1.7,\ymin+2.7) -- (\xmin+3.8,\ymin+2.7);
\draw (\xmin+2.75,\ymin+2.7) node[above,scale=0.7] {$\gamma_1 n$};
\draw (\xmin+2.75,\ymin+5) node[scale=2]{$A_1$};
%graph
\draw[very thick] (\xmin+0.5,\ymax-5) -- (\xmin+1.5,\ymax-1.5) -- (\xmin+4,\ymax-1.5) -- (0,\ymin+2.5);
\end{tikzpicture}
}%
\subfloat[\label{fig:case2area}]{
\begin{tikzpicture}[scale=0.55]
%grid
\def\xmin{-5.5}
\def\xmax{5.5}
\def\ymin{-4.5}
\def\ymax{4.5}
\def\scale{0.7}
%axes
\draw[very thick, <->] (\xmin,\ymin+2.5)--(0.5,\ymin+2.5) node[right,scale=\scale]{$y$};
\draw[very thick, <->] (\xmin+0.5,\ymin+2)--(\xmin+0.5,\ymax) node[above,scale=\scale]{pairs};
%label
\draw (0,\ymin+2.5) node[below,scale=\scale] {$\frac{n}{4}$};
\draw (0,\ymin+2.4) -- (0,\ymin+2.6);
\draw (\xmin+0.5,\ymax-1.5) node[left,scale=\scale] {$\mu_{BB}$};
\draw (\xmin+0.4,\ymax-1.5) -- (\xmin+0.6,\ymax-1.5);
\draw (\xmin+0.5,\ymax-5) node[left,scale=\scale] {$\mu_{BB}-\mu_{RR}$};
\draw (\xmin+0.4,\ymax-5) -- (\xmin+0.6,\ymax-5);
%shading
\draw[black, fill=lightgray] (\xmin+1,\ymin+2.5) -- (\xmin+1,1.2) -- (\xmin+1.5,\ymax-1.5) -- (\xmin+4,\ymax-1.5) -- (\xmin+4.5,1.2) -- (\xmin+4.5,\ymin+2.5) -- (\xmin+1,\ymin+2.5);
\draw [|-|, thick] (\xmin+1,\ymin+2.7) -- (\xmin+4.5,\ymin+2.7);
\draw (\xmin+2.75,\ymin+2.7) node[above,scale=0.7] {$\gamma_1 n$};
\draw (\xmin+2.75,\ymin+5) node[scale=2]{$A_1$};
%graph
\draw[very thick] (\xmin+0.5,\ymax-5) -- (\xmin+1.5,\ymax-1.5) -- (\xmin+4,\ymax-1.5) -- (0,\ymin+2.5);
\draw[dashed] (\xmin+1,1.2) -- (\xmin+1, \ymax-1.5) -- (\xmin+1.5,\ymax-1.5);
\draw[dashed] (\xmin+4,\ymax-1.5) -- (\xmin+4.5,\ymax-1.5) -- (\xmin+4.5,1.2);
\end{tikzpicture}
}%
\subfloat[\label{fig:case3area}]{
\begin{tikzpicture}[scale=0.55]
%grid
\def\xmin{-5.5}
\def\xmax{5.5}
\def\ymin{-4.5}
\def\ymax{4.5}
\def\scale{0.7}
%axes
\draw[very thick, <->] (\xmin,\ymin+2.5)--(0.5,\ymin+2.5) node[right,scale=\scale]{$y$};
\draw[very thick, <->] (\xmin+0.5,\ymin+2)--(\xmin+0.5,\ymax) node[above,scale=\scale]{pairs};
%label
\draw (0,\ymin+2.5) node[below,scale=\scale] {$\frac{n}{4}$};
\draw (0,\ymin+2.4) -- (0,\ymin+2.6);
\draw (\xmin+0.5,\ymax-1.5) node[left,scale=\scale] {$\mu_{BB}$};
\draw (\xmin+0.4,\ymax-1.5) -- (\xmin+0.6,\ymax-1.5);
\draw (\xmin+0.5,\ymax-5) node[left,scale=\scale] {$\mu_{BB}-\mu_{RR}$};
\draw (\xmin+0.4,\ymax-5) -- (\xmin+0.6,\ymax-5);
%shading
%shading
\draw[black, fill=lightgray] (\xmin+0.5,\ymin+2.5) -- (\xmin+0.5,\ymax-5) -- (\xmin+1.5,\ymax-1.5) -- (\xmin+4,\ymax-1.5) -- (\xmin+5.1,-0.8) -- (\xmin+5.1,\ymin+2.5) -- (\xmin+0.5,\ymin+2.5);
\draw [|-|, thick] (\xmin+0.5,\ymin+2.7) -- (\xmin+5.1,\ymin+2.7);
\draw (\xmin+2.75,\ymin+2.7) node[above,scale=0.7] {$\gamma_1 n$};
\draw (\xmin+2.75,\ymin+5) node[scale=2]{$A_1$};
%graph
\draw[very thick] (\xmin+0.5,\ymax-5) -- (\xmin+1.5,\ymax-1.5) -- (\xmin+4,\ymax-1.5) -- (0,\ymin+2.5);
\draw[dashed] (\xmin+0.5,\ymax-1.5) -- (\xmin+1.5,\ymax-1.5);
\draw[dashed] (\xmin+4,\ymax-1.5) -- (\xmin+5.1,\ymax-1.5) -- (\xmin+5.1,\ymin+2.5);
\end{tikzpicture}
}
\caption{The Three Cases of Proposition \ref{N2}}
\label{fig:3cases}
\end{center}
\end{figure}

\noindent We complete the proof of this proposition as follows. For each case, we write $A_1$ in terms of the variables $\mu_R, \mu_B, \gamma$ and $\gamma_1.$ Then, we optimize $\gamma, \gamma_1$ with respect to $\mu_{B_1}, \mu_{R_1}$ and $\mu_{R_2}$.  Finally, we use the optimal $\gamma$ and $\gamma_1$ values to maximize 
\[ \Delta := \dfrac{1}{2}\left( \dfrac{\mu_R\mu_B}{2}+\mu_R\mu_{B_1}
+\mu_B\mu_{R_1}+2A_1 \right). \]

\noindent Denote this maximum to be $\Delta_{max}$.  Propositions \ref{nonmono2} and \ref{d2bound} show that $\Delta_{max}$ will be the upper bound for $|\mathcal{N}(n)|$. The optimization of $\Delta$ has been done using Maple and for curious readers, the code can be found at Thanatipanonda's website. We note that $\Delta$ can ultimately be written as a function of only two variables $\mu_R$ and $\mu_{R_1}$. In our calculations, we use the following lower bound of $\gamma n$ and upper bound of $\gamma_1 n$:
\begin{align*}
\gamma_1 n&\leq \mbox{ min}(\mu_{R_1}, \mu_{B_1})\\
\gamma n&\geq |\mu_{R_1}-\mu_{R_2}|
\end{align*}

\noindent Maple's current technology does not allow us to optimize with absolute value and minimum functions.  Thus, we separate each case into the following pieces.  The subcases are summarized in the table in Figure \ref{fig:global}.  It is important to note that subcase D can be ignored in our calculation because it only produces one pair, namely $\mu_{R_1}=\frac{n}{4}$ and $\mu_{R}=\frac{n}{2}$ (recall that $\mu_R \leq \frac{n}{2}$).
\vspace{0.5cm}

\begin{figure}[ht]
\begin{centering}
\adjustbox{max width=\linewidth, keepaspectratio}{
\begin{tabular}{|c|c|}
%\hline
%\multicolumn{2}{|l|}{$\mu_B \geq \mu_R, \;\ 0\leq\mu_{R_1}\leq\mu_R$}\\
\hline
Subcase & Conditions on $\mu_{R_1}, \mu_{B_1}$ and $\mu_{R_2}$ \\
\hline
A &  $\mu_{B_1}\leq\mu_{R_1}$ and $\mu_{R_1}\geq \mu_{R_2}$\\
\hline
B &  $\mu_{B_1} \geq \mu_{R_1}$ and $\mu_{R_1}\geq \mu_{R_2}$ \\
\hline
C &  $\mu_{B_1} \geq \mu_{R_1}$ and $\mu_{R_1}\leq \mu_{R_2}$ \\
\hline
D &  $\mu_{B_1}\leq \mu_{R_1}$ and $\mu_{R_1}\leq \mu_{R_2}$\\
\hline
\end{tabular}}
\caption{The Four Subcases}
\label{fig:global}
\end{centering}
\end{figure}

\noindent\textbf{Case 1:} $\gamma_1 n<\dfrac{\gamma n}{2}$.  This case is illustrated in Figure \ref{fig:case1area}.
\begin{align*}
A_1&=\mu_{BB}\cdot\gamma_1 n\\
&=\frac{\mu_B-\gamma n}{2}\cdot\gamma_1 n.
\end{align*}

\noindent In order to maximize $A_1$, we maximize $\gamma_1 n$ and minimize $\gamma n$.  To be able to determine the values of $\gamma_1$ and $\gamma$, we consider two further subcases as follows:
\vspace{0.4cm}

\noindent \textbf{Case 1.1:} $\mbox{min}\left(\mu_{R_1},\mu_{B_1}\right)<\frac{|\mu_{R_1}-\mu_{R_2}|}{2}$
\vspace{0.2cm}

\noindent The optimal values of $(\gamma_1 n, \gamma n)$ is 
$(\mbox{min}\left(\mu_{R_1},\mu_{B_1}\right) , |\mu_{R_1}-\mu_{R_2}|)$ as shown in Figure \ref{fig:case1.1graph}.
\vspace{0.2cm}

\noindent \textbf{Case 1.2:} $\mbox{min}\left(\mu_{R_1},\mu_{B_1}\right)\geq\frac{|\mu_{R_1}-\mu_{R_2}|}{2}$
\vspace{0.2cm}

\noindent The optimal values of $(\gamma_1 n, \gamma n)$ lies on the line $\gamma n = 2\gamma_1 n$  as shown in Figure \ref{fig:case1.2graph}. This gives 
\[ A_1 =\frac{\left(\mu_B-\gamma n\right)\cdot\gamma n}{4}.\]
Thus, $A_1$ attains its maximum value at $\gamma n =\frac{\mu_B}{2}$.
\vspace{0.4cm}

\noindent The calculations for $\Delta_{max}$ for all subcases are summarized in the table in Figure \ref{fig:caseonetable} and the admissible regions for subcases A, B, and C are shown in Figure \ref{fig:caseonegraph}.  In this case, $\Delta_{max}$ occurs under subcase A.
\vspace{0.5cm}

\begin{figure}[ht]
\begin{center}
\subfloat[\label{fig:case1.1graph}]{
\begin{tikzpicture}[scale=0.5]
%grid
\def\xmin{-4.5}
\def\xmax{4.5}
\def\ymin{-4.5}
\def\ymax{4.5}
%axes
\draw[very thick, <->] (\xmin,\ymin+0.5)--(\xmax,\ymin+0.5) node[right,scale=1]{$\gamma_1 n$};
\draw[very thick, <->] (\xmin+0.5,\ymin)--(\xmin+0.5,\ymax) node[above,scale=1]{$\gamma n$};
%label axes
\draw (\xmin+0.5,\ymax-0.5) node[left,scale=1] {$\frac{n}{2}$};
\draw (\xmin+0.4,\ymax-0.5) -- (\xmin+0.6,\ymax-0.5);
\draw (\xmin+0.5,1.5) node[left,scale=1] {$\mu_R$};
\draw (\xmin+0.4,1.5) -- (\xmin+0.6,1.5);
\draw (\xmin+0.5,-1) node[left,scale=1] {$|\mu_{R_1}-\mu_{R_2}|$};
\draw (\xmin+0.4,-1) -- (\xmin+0.6,-1);
\draw (\xmin+1.5,\ymin+0.5) node[below,scale=1] {min$(\mu_{R_1},\mu_{B_1})$};
\draw (\xmin+1.5,\ymin+0.4) -- (\xmin+1.5,\ymin+0.6);
%main line
\draw[very thick] (\xmin+0.5,\ymin+0.5) -- (0,\ymax-0.5);
\draw (0,\ymax-0.5) node[right,scale=1] {$\gamma n=2\gamma_1 n$};
%rectangle and optimal point
\draw[fill=lightgray] (\xmin+0.5,-1) rectangle (\xmin+1.5,1.5);
\draw[fill] (\xmin+1.5,-1) circle [radius=0.1];
\draw[dashed] (\xmin+1.5,-1) -- (\xmin+1.5,\ymin+0.5);
\end{tikzpicture}
}%
\subfloat[\label{fig:case1.2graph}]{
\begin{tikzpicture}[scale=0.5]
%grid
\def\xmin{-4.5}
\def\xmax{4.5}
\def\ymin{-4.5}
\def\ymax{4.5}
%axes
\draw[very thick, <->] (\xmin,\ymin+0.5)--(\xmax,\ymin+0.5) node[right,scale=1]{$\gamma_1 n$};
\draw[very thick, <->] (\xmin+0.5,\ymin)--(\xmin+0.5,\ymax) node[above,scale=1]{$\gamma n$};
%label axes
\draw (\xmin+0.5,\ymax-0.5) node[left,scale=1] {$\frac{n}{2}$};
\draw (\xmin+0.4,\ymax-0.5) -- (\xmin+0.6,\ymax-0.5);
\draw (\xmin+0.5,1.5) node[left,scale=1] {$\mu_R$};
\draw (\xmin+0.4,1.5) -- (\xmin+0.6,1.5);
\draw (\xmin+0.5,-1) node[left,scale=1] {$|\mu_{R_1}-\mu_{R_2}|$};
\draw (\xmin+0.4,-1) -- (\xmin+0.6,-1);
\draw (\xmin+2.7,\ymin+0.5) node[below,scale=1] {min$(\mu_{R_1},\mu_{B_1})$};
\draw (\xmin+2.7,\ymin+0.4) -- (\xmin+2.7,\ymin+0.6);
%main line
\draw[very thick] (\xmin+0.5,\ymin+0.5) -- (0,\ymax-0.5);
\draw (0,\ymax-0.5) node[right,scale=1] {$\gamma n=2\gamma_1 n$};
%region
\draw[black, fill=lightgray] (\xmin+0.5,-1) -- (\xmin+0.5,1.5) -- (\xmin+2.7,1.5) -- (\xmin+2.7,0.45) -- (\xmin+1.98,-1) -- (\xmin+0.5,-1);
\draw[fill] (\xmin+2.7,0.45) circle [radius=0.1];
\draw[fill] (\xmin+1.98,-1) circle [radius=0.1];
\draw[dashed] (\xmin+2.7,0.45) -- (\xmin+2.7,\ymin+0.5);
\end{tikzpicture}
}%
\caption{Finding the optimal values of $\gamma_1 n$ and $\gamma n$ in Case 1}
\end{center}
\end{figure}
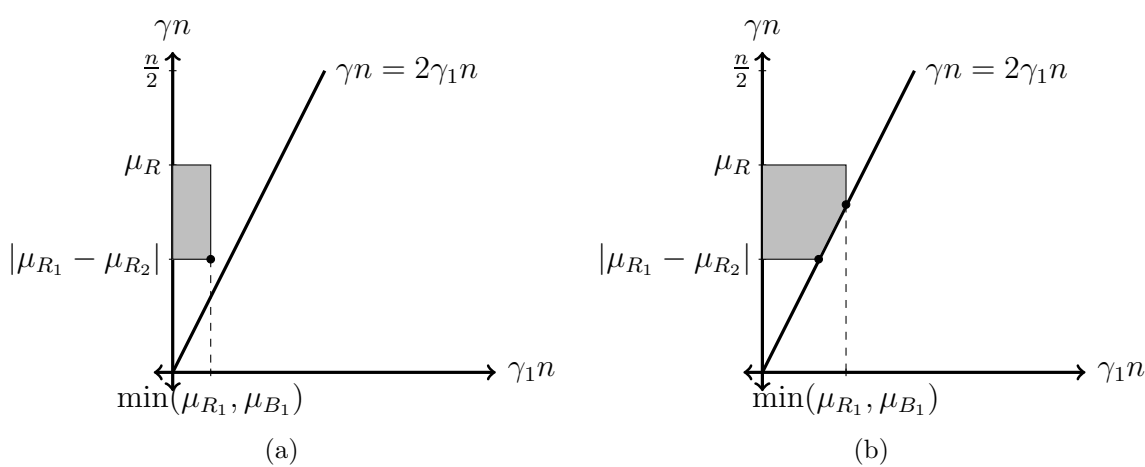

\begin{figure}[ht]
\begin{centering}
\adjustbox{max width=\linewidth, keepaspectratio}{
\begin{tabular}{|c|c|c|c|c|c|}
\hline
Case & Subcase & Optimal $\gamma_1 n$& Optimal $\gamma n$ & $\Delta_{max}$ & $(\mu_{R_1},\mu_R)$\\
\hline
&A&$\mu_{B_1}$&$\mu_{R_1}-\mu_{R_2}$&$\frac{2n^2}{9}$&$\left(\frac{n}{3},\frac{n}{3}\right)$\\
1.1&B&$\mu_{R_1}$&$\mu_{R_1}-\mu_{R_2}$&0&(0,0)\\
&C&$\mu_{R_1}$&$\mu_{R_2}-\mu_{R_1}$&$\frac{13n^2}{64}$&$\left(\frac{n}{8},\frac{n}{2}\right)$\\
\hline
&A&$ \frac{\mu_B}{4}$&$\frac{\mu_B}{2}$&$\frac{2n^2}{9}$&$\left(\frac{n}{3},\frac{n}{3}\right)$\\
1.2&B&$  \frac{\mu_B}{4}$&$\frac{\mu_B}{2}$&$\frac{5n^2}{24}$&$\left(\frac{n}{4},\frac{n}{3}\right)$\\
&C&$  \frac{\mu_B}{4}$&$\frac{\mu_B}{2}$&$\frac{9n^2}{44}$&$\left(\frac{5n}{22},\frac{5n}{11}\right)$\\
\hline
\end{tabular}}
\caption{Results for Case 1}
\label{fig:caseonetable}
\end{centering}
\end{figure}

\newpage
\noindent\textbf{Case 2:} $\frac{\gamma n}{2} \leq \gamma_1 n<\frac{\mu_R}{2}$.  This case is illustrated in Figure \ref{fig:case2area}.
\vspace{0.3cm}

\noindent In this case, we take the area of the rectangle but subtract those pairs that are outside of the region.  Thus, we have that:
\begin{align*}
A_1 &=\mu_{BB}\cdot\gamma_1 n-\frac{1}{2}\left(\gamma_1 n - \frac{\gamma n}{2}\right)^2\\
&=\frac{(\mu_B-\gamma_1 n)\cdot\gamma_1 n}{2}-\frac{(\gamma n)^2}{8}.
\end{align*}

\noindent In order to maximize $A_1$, we want to make $\gamma_1 n$ as close to $\frac{\mu_B}{2}$ as possible and minimize $\gamma n$. We break this case down into subcases depending on whether or not the upper bound of $\gamma_1 n$ is less than $\frac{\mu_R}{2}$.
\vspace{0.4cm} 

\noindent \textbf{Case 2.1} $\mbox{min}\left(\mu_{R_1},\mu_{B_1}\right)<\frac{\mu_R}{2}$
\vspace{0.2cm}

\noindent The optimal value of $(\gamma_1 n, \gamma n)$ is 
$(\mbox{min}\left(\mu_{R_1},\mu_{B_1}\right) , |\mu_{R_1}-\mu_{R_2}|)$.
\vspace{0.2cm}

\noindent \textbf{Case 2.2} $\mbox{min}\left(\mu_{R_1},\mu_{B_1}\right)\geq \frac{\mu_R}{2}$
\vspace{0.2cm}

\noindent The optimal value of $(\gamma_1 n, \gamma n)$ is  
$(\frac{\mu_R}{2} , |\mu_{R_1}-\mu_{R_2}|)$.
\vspace{0.4cm}

\noindent The calculations for $\Delta_{max}$ for all subcases are summarized in the table in Figure \ref{fig:casetwotable} and the admissible regions for subcases A, B, and C are shown in Figure \ref{fig:casetwograph}.  In this case, $\Delta_{max}$ occurs under subcase A.
\vspace{0.5cm}

\begin{figure}[ht]
\begin{centering}
\adjustbox{max width=\linewidth, keepaspectratio}{
\begin{tabular}{|c|c|c|c|c|c|}
\hline
Case & Subcase & Optimal $\gamma_1 n$& Optimal $\gamma n$ & $\Delta_{max}$ & $(\mu_{R_1},\mu_R)$\\
\hline
&A&$\mu_{B_1}$&$\mu_{R_1}-\mu_{R_2}$&$\frac{9n^2}{40}$&$\left(\frac{3n}{10},\frac{2n}{5}\right)$\\
2.1&B&$\mu_{R_1}$&$\mu_{R_1}-\mu_{R_2}$&$\frac{2n^2}{9}$&$\left(\frac{2n}{9},\frac{4n}{9}\right)$\\
&C&$\mu_{R_1}$&$\mu_{R_2}-\mu_{R_1}$&$\frac{2n^2}{9}$&$\left(\frac{2n}{9},\frac{4n}{9}\right)$\\
\hline
&A&$\frac{\mu_R}{2}$&$\mu_{R_1}-\mu_{R_2}$&$\frac{9n^2}{40}$&$\left(\frac{3n}{10},\frac{2n}{5}\right)$\\
2.2&B&$\frac{\mu_R}{2}$&$\mu_{R_1}-\mu_{R_2}$&$\frac{43n^2}{192}$&$\left(\frac{n}{4},\frac{5n}{12}\right)$\\
&C&$\frac{\mu_R}{2}$&$\mu_{R_2}-\mu_{R_1}$&$\frac{2n^2}{9}$&$\left(\frac{2n}{9},\frac{4n}{9}\right)$\\
\hline
\end{tabular}}
\caption{Results for Case 2}
\label{fig:casetwotable}
\end{centering}
\end{figure}

\noindent\textbf{Case 3:} $\frac{\mu_R}{2}<\gamma_1 n$. This case is illustrated in Figure \ref{fig:case3area}.
\vspace{0.3cm}

\noindent In this final case, $A_1$ indicated by nearly the entire region under the graph. We take the area of the rectangle but subtract those pairs that are outside of the region.  Thus, we have that
\begin{align*}
A_1 &=\mu_{BB}\cdot\gamma_1 n-\left(\frac{\mu_{RR}}{2}\right)^2-\left(\gamma_1 n-\frac{\mu_{RR}}{2}-\frac{\gamma n}{2}\right)^2\\
&=\left(\frac{n}{2}-\gamma_1 n\right)\cdot\gamma_1 n-\frac{1}{4}\left((\gamma n+\mu_{RR})^2+\mu_{RR}^2\right)\\
&=\left(\frac{n}{2}-\gamma_1 n\right)\cdot\gamma_1 n-\frac{(\gamma n)^2}{8}-\frac{\mu_{R}^2}{8}.\\
\end{align*}

\noindent In order to maximize $A_1$, we want to make $\gamma_1 n$ as to close to $\frac{n}{4}$ as possible and minimize $\gamma n$. The calculations for $\Delta_{max}$ for all subcases are summarized in the table in Figure \ref{fig:casethreetable} and the admissible regions for subcases A, B, and C are shown in Figure \ref{fig:casethreegraph}.  Once again, $\Delta_{max}$ occurs under subcase A.
\vspace{0.5cm}

\begin{figure}[ht]
\begin{centering}
\adjustbox{max width=\linewidth, keepaspectratio}{
\begin{tabular}{|c|c|c|c|c|c|}
\hline
Case & Subcase & Optimal $\gamma_1 n$& Optimal $\gamma n$ & $\Delta_{max}$ & $(\mu_{R_1},\mu_R)$\\
\hline
&A&$\mu_{B_1}$&$\mu_{R_1}-\mu_{R_2}$&$\frac{5n^2}{22}$&$\left(\frac{3n}{11},\frac{4n}{11}\right)$\\
3&B&$\mu_{R_1}$&$\mu_{R_1}-\mu_{R_2}$&$\frac{29n^2}{128}$&$\left(\frac{n}{4},\frac{3n}{8}\right)$\\
&C&$\mu_{R_1}$&$\mu_{R_2}-\mu_{R_1}$&$\frac{2n^2}{9}$&$\left(\frac{2n}{9},\frac{4n}{9}\right)$\\
\hline
\end{tabular}}
\caption{Results for Case 3}
\label{fig:casethreetable}
\end{centering}
\end{figure}

\noindent In all three cases, we can see that $\Delta_{max}$ occurs in Case 3 when $\mu_{R_1}=\frac{3n}{11}$ and $\mu_R=\frac{4n}{11}$ with $\Delta_{max} = \frac{5n^2}{22}+\mathcal{O}(n)$.

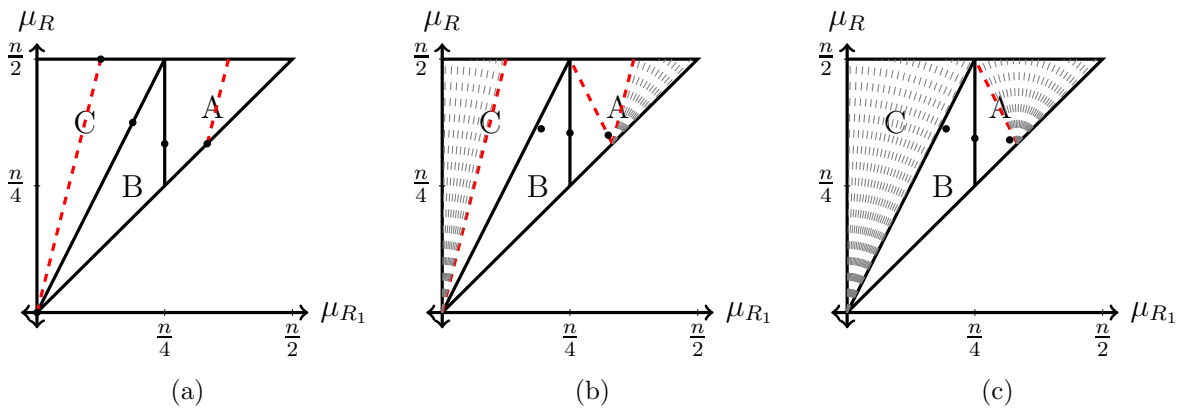
\begin{figure}[ht]
\begin{center}
\subfloat[\label{fig:caseonegraph}]{
\begin{tikzpicture}[scale=0.42]
%grid
\def\xmin{-4.5}
\def\xmax{4.5}
\def\ymin{-4.5}
\def\ymax{4.5}
%axes
\draw[very thick, <->] (\xmin,\ymin+0.5)--(\xmax,\ymin+0.5) node[right,scale=1]{$\mu_{R_1}$};
\draw[very thick, <->] (\xmin+0.5,\ymin)--(\xmin+0.5,\ymax) node[above,scale=1]{$\mu_{R}$};
%label axes
\draw (0,\ymin+0.5) node[below,scale=1] {$\frac{n}{4}$};
\draw (0,\ymin+0.4) -- (0,\ymin+0.6);
\draw (\xmax-0.5,\ymin+0.5) node[below,scale=1] {$\frac{n}{2}$};
\draw (\xmax-0.5,\ymin+0.4) -- (\xmax-0.5,\ymin+0.6);
\draw (\xmin+0.5,0) node[left,scale=1] {$\frac{n}{4}$};
\draw (\xmin+0.4,0) -- (\xmin+0.6,0);
\draw (\xmin+0.5,\ymax-0.5) node[left,scale=1] {$\frac{n}{2}$};
\draw (\xmin+0.4,\ymax-0.5) -- (\xmin+0.6,\ymax-0.5);
%regions
\draw[very thick] (\xmin+0.5,\ymin+0.5) -- (\xmax-0.5,\ymax-0.5) -- (\xmin+0.5,\ymax-0.5);
\draw[very thick] (0,0) -- (0,\ymax-0.5);
\draw[very thick] (\xmin+0.5,\ymin+0.5) -- (0,\ymax-0.5);
\draw (\xmax-3,\ymax-2) node[scale=1]{A};
\draw (\xmin+3.5,\ymin+4.5) node[scale=1]{B};
\draw (\xmin+2,\ymax-2.5) node[scale=1]{C};
%dividing lines
\draw[very thick, red, dashed] (\xmin+0.5,\ymin+0.5) -- (\xmin+2.5,\ymax-0.5);
%\draw[very thick, red] (\xmin+0.5,\ymin+0.5) -- (0,\ymax-0.5);
\draw[very thick, red, dashed] (1.33,1.33) -- (2,\ymax-0.5);
%maximum
\draw[fill] (\xmin+0.5,\ymin+0.5) circle [radius=0.1];
\draw[fill] (-2,\ymax-0.5) circle [radius=0.1];
\draw[fill] (-1,2) circle [radius=0.1];
\draw[fill] (0,1.33) circle [radius=0.1];
\draw[fill] (1.33,1.33) circle [radius=0.1];
\end{tikzpicture}
}%
\subfloat[\label{fig:casetwograph}]{
\begin{tikzpicture}[scale=0.42]
%grid
\def\xmin{-4.5}
\def\xmax{4.5}
\def\ymin{-4.5}
\def\ymax{4.5}
%axes
\draw[very thick, <->] (\xmin,\ymin+0.5)--(\xmax,\ymin+0.5) node[right,scale=1]{$\mu_{R_1}$};
\draw[very thick, <->] (\xmin+0.5,\ymin)--(\xmin+0.5,\ymax) node[above,scale=1]{$\mu_{R}$};
%label axes
\draw (0,\ymin+0.5) node[below,scale=1] {$\frac{n}{4}$};
\draw (0,\ymin+0.4) -- (0,\ymin+0.6);
\draw (\xmax-0.5,\ymin+0.5) node[below,scale=1] {$\frac{n}{2}$};
\draw (\xmax-0.5,\ymin+0.4) -- (\xmax-0.5,\ymin+0.6);
\draw (\xmin+0.5,0) node[left,scale=1] {$\frac{n}{4}$};
\draw (\xmin+0.4,0) -- (\xmin+0.6,0);
\draw (\xmin+0.5,\ymax-0.5) node[left,scale=1] {$\frac{n}{2}$};
\draw (\xmin+0.4,\ymax-0.5) -- (\xmin+0.6,\ymax-0.5);
%regions
\draw[very thick] (\xmin+0.5,\ymin+0.5) -- (\xmax-0.5,\ymax-0.5) -- (\xmin+0.5,\ymax-0.5);
\draw[very thick] (0,0) -- (0,\ymax-0.5);
\draw[very thick] (\xmin+0.5,\ymin+0.5) -- (0,\ymax-0.5);
\draw (\xmax-3,\ymax-2) node[scale=1]{A};
\draw (\xmin+3.5,\ymin+4.5) node[scale=1]{B};
\draw (\xmin+2,\ymax-2.5) node[scale=1]{C};
%dividing lines
\draw[very thick, red, dashed] (\xmin+0.5,\ymin+0.5) -- (\xmin+2.5,\ymax-0.5);
\draw[very thick, red, dashed] (1.33,1.33) -- (2,\ymax-0.5);
\draw[very thick, red, dashed] (1.33,1.33) -- (0,\ymax-0.5);
\draw[gray, dashed] (\xmin+0.5,\ymin+0.5) -- (\xmin+0.6,\ymax-0.5);
\draw[gray, dashed] (\xmin+0.5,\ymin+0.5) -- (\xmin+0.8,\ymax-0.5);
\draw[gray, dashed] (\xmin+0.5,\ymin+0.5) -- (\xmin+1.0,\ymax-0.5);
\draw[gray, dashed] (\xmin+0.5,\ymin+0.5) -- (\xmin+1.2,\ymax-0.5);
\draw[gray, dashed] (\xmin+0.5,\ymin+0.5) -- (\xmin+1.4,\ymax-0.5);
\draw[gray, dashed] (\xmin+0.5,\ymin+0.5) -- (\xmin+1.6,\ymax-0.5);
\draw[gray, dashed] (\xmin+0.5,\ymin+0.5) -- (\xmin+1.8,\ymax-0.5);
\draw[gray, dashed] (\xmin+0.5,\ymin+0.5) -- (\xmin+2.0,\ymax-0.5);
\draw[gray, dashed] (\xmin+0.5,\ymin+0.5) -- (\xmin+2.2,\ymax-0.5);
\draw[gray, dashed] (\xmin+0.5,\ymin+0.5) -- (\xmin+2.4,\ymax-0.5);
\draw[gray, dashed] (1.33,1.33) -- (2.2,\ymax-0.5);
\draw[gray, dashed] (1.33,1.33) -- (2.4,\ymax-0.5);
\draw[gray, dashed] (1.33,1.33) -- (2.6,\ymax-0.5);
\draw[gray, dashed] (1.33,1.33) -- (2.8,\ymax-0.5);
\draw[gray, dashed] (1.33,1.33) -- (3.0,\ymax-0.5);
\draw[gray, dashed] (1.33,1.33) -- (3.2,\ymax-0.5);
\draw[gray, dashed] (1.33,1.33) -- (3.4,\ymax-0.5);
\draw[gray, dashed] (1.33,1.33) -- (3.6,\ymax-0.5);
\draw[gray, dashed] (1.33,1.33) -- (3.8,\ymax-0.5);
%maximum
\draw[fill] (1.2,1.6) circle [radius=0.1];
\draw[fill] (-0.9,1.8) circle [radius=0.1];
\draw[fill] (0,1.67) circle [radius=0.1];
\end{tikzpicture}
}%
\subfloat[\label{fig:casethreegraph}]{
\begin{tikzpicture}[scale=0.42]
%grid
\def\xmin{-4.5}
\def\xmax{4.5}
\def\ymin{-4.5}
\def\ymax{4.5}
%axes
\draw[very thick, <->] (\xmin,\ymin+0.5)--(\xmax,\ymin+0.5) node[right,scale=1]{$\mu_{R_1}$};
\draw[very thick, <->] (\xmin+0.5,\ymin)--(\xmin+0.5,\ymax) node[above,scale=1]{$\mu_{R}$};
%label axes
\draw (0,\ymin+0.5) node[below,scale=1] {$\frac{n}{4}$};
\draw (0,\ymin+0.4) -- (0,\ymin+0.6);
\draw (\xmax-0.5,\ymin+0.5) node[below,scale=1] {$\frac{n}{2}$};
\draw (\xmax-0.5,\ymin+0.4) -- (\xmax-0.5,\ymin+0.6);
\draw (\xmin+0.5,0) node[left,scale=1] {$\frac{n}{4}$};
\draw (\xmin+0.4,0) -- (\xmin+0.6,0);
\draw (\xmin+0.5,\ymax-0.5) node[left,scale=1] {$\frac{n}{2}$};
\draw (\xmin+0.4,\ymax-0.5) -- (\xmin+0.6,\ymax-0.5);
%regions
\draw[very thick] (\xmin+0.5,\ymin+0.5) -- (\xmax-0.5,\ymax-0.5) -- (\xmin+0.5,\ymax-0.5);
\draw[very thick] (0,0) -- (0,\ymax-0.5);
\draw[very thick] (\xmin+0.5,\ymin+0.5) -- (0,\ymax-0.5);
\draw (\xmax-3.7,\ymax-2) node[scale=1]{A};
\draw (\xmin+3.5,\ymin+4.5) node[scale=1]{B};
\draw (\xmin+2,\ymax-2.5) node[scale=1]{C};
%dividing lines
\draw[very thick, red, dashed] (1.33,1.33) -- (0,\ymax-0.5);
\draw[gray, dashed] (\xmin+0.5,\ymin+0.5) -- (\xmin+0.6,\ymax-0.5);
\draw[gray, dashed] (\xmin+0.5,\ymin+0.5) -- (\xmin+0.8,\ymax-0.5);
\draw[gray, dashed] (\xmin+0.5,\ymin+0.5) -- (\xmin+1.0,\ymax-0.5);
\draw[gray, dashed] (\xmin+0.5,\ymin+0.5) -- (\xmin+1.2,\ymax-0.5);
\draw[gray, dashed] (\xmin+0.5,\ymin+0.5) -- (\xmin+1.4,\ymax-0.5);
\draw[gray, dashed] (\xmin+0.5,\ymin+0.5) -- (\xmin+1.6,\ymax-0.5);
\draw[gray, dashed] (\xmin+0.5,\ymin+0.5) -- (\xmin+1.8,\ymax-0.5);
\draw[gray, dashed] (\xmin+0.5,\ymin+0.5) -- (\xmin+2.0,\ymax-0.5);
\draw[gray, dashed] (\xmin+0.5,\ymin+0.5) -- (\xmin+2.2,\ymax-0.5);
\draw[gray, dashed] (\xmin+0.5,\ymin+0.5) -- (\xmin+2.4,\ymax-0.5);
\draw[gray, dashed] (\xmin+0.5,\ymin+0.5) -- (\xmin+2.6,\ymax-0.5);
\draw[gray, dashed] (\xmin+0.5,\ymin+0.5) -- (\xmin+2.8,\ymax-0.5);
\draw[gray, dashed] (\xmin+0.5,\ymin+0.5) -- (\xmin+3.0,\ymax-0.5);\draw[gray, dashed] (\xmin+0.5,\ymin+0.5) -- (\xmin+3.2,\ymax-0.5);\draw[gray, dashed] (\xmin+0.5,\ymin+0.5) -- (\xmin+3.4,\ymax-0.5);\draw[gray, dashed] (\xmin+0.5,\ymin+0.5) -- (\xmin+3.6,\ymax-0.5);\draw[gray, dashed] (\xmin+0.5,\ymin+0.5) -- (\xmin+3.8,\ymax-0.5);
\draw[gray, dashed] (\xmin+0.5,\ymin+0.5) -- (\xmin+4.0,\ymax-0.5);
\draw[gray, dashed] (\xmin+0.5,\ymin+0.5) -- (\xmin+4.2,\ymax-0.5);
\draw[gray, dashed] (\xmin+0.5,\ymin+0.5) -- (\xmin+4.4,\ymax-0.5);
\draw[gray, dashed] (1.33,1.33) -- (0.2,\ymax-0.5);
\draw[gray, dashed] (1.33,1.33) -- (0.4,\ymax-0.5);
\draw[gray, dashed] (1.33,1.33) -- (0.6,\ymax-0.5);
\draw[gray, dashed] (1.33,1.33) -- (0.8,\ymax-0.5);
\draw[gray, dashed] (1.33,1.33) -- (1.0,\ymax-0.5);
\draw[gray, dashed] (1.33,1.33) -- (1.2,\ymax-0.5);
\draw[gray, dashed] (1.33,1.33) -- (1.4,\ymax-0.5);
\draw[gray, dashed] (1.33,1.33) -- (1.6,\ymax-0.5);
\draw[gray, dashed] (1.33,1.33) -- (1.8,\ymax-0.5);
\draw[gray, dashed] (1.33,1.33) -- (2.0,\ymax-0.5);
\draw[gray, dashed] (1.33,1.33) -- (2.2,\ymax-0.5);
\draw[gray, dashed] (1.33,1.33) -- (2.4,\ymax-0.5);
\draw[gray, dashed] (1.33,1.33) -- (2.6,\ymax-0.5);
\draw[gray, dashed] (1.33,1.33) -- (2.8,\ymax-0.5);
\draw[gray, dashed] (1.33,1.33) -- (3.0,\ymax-0.5);
\draw[gray, dashed] (1.33,1.33) -- (3.2,\ymax-0.5);
\draw[gray, dashed] (1.33,1.33) -- (3.4,\ymax-0.5);
\draw[gray, dashed] (1.33,1.33) -- (3.6,\ymax-0.5);
\draw[gray, dashed] (1.33,1.33) -- (3.8,\ymax-0.5);
%maximum
\draw[fill] (1.09,1.45) circle [radius=0.1];
\draw[fill] (-0.9,1.8) circle [radius=0.1];
\draw[fill] (0,1.5) circle [radius=0.1];
\end{tikzpicture}
}
\end{center}
\caption{Admissible Regions of Cases in Proposition \ref{N2}}
\end{figure}

\end{proof}

\noindent And now, we are ready to present the main result of this paper.
\vspace{0.2cm}

\begin{theorem}
Over all 2-colorings of $[1,n]$, the minimum number of monochromatic triples satisfying $x+2y=z$ is $\frac{n^2}{44}+\mathcal{O}(n)$.
\end{theorem}

\begin{proof}

\noindent An upper bound of the minimum can be obtained from a coloring on $[1,n]$.  We color $\left[R^{3n/11},B^{7n/11},R^{n/11}\right]$ as illustrated in Figure \ref{fig:optimalcoloring2}. This solution was discovered in Butler, Costello, and Graham \cite{butler} and in Thanathipanonda \cite{aek}.
\vspace{0.5cm}

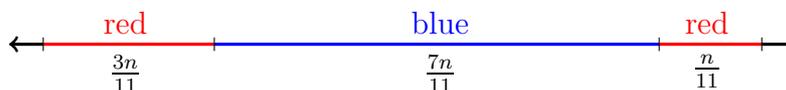
\begin{figure}[ht]
\begin{center}
\begin{tikzpicture}[scale=0.9]
%grid
\def\xmin{-0.5}
\def\xmax{11.5}
\def\ymin{-1}
\def\ymax{1}
%axes
\draw[very thick, <-] (\xmin,\ymin+0.5)--(\xmin+0.5,\ymin+0.5);
\draw[very thick, red, -] (\xmin+0.5,\ymin+0.5)--(\xmin+3,\ymin+0.5);
\draw[very thick, blue, -] (\xmin+3,\ymin+0.5)--(\xmin+9.5,\ymin+0.5);
\draw[very thick, red, -] (\xmin+9.5,\ymin+0.5)--(\xmin+11,\ymin+0.5);
\draw[very thick, ->] (\xmin+11,\ymin+0.5)--(\xmin+11.5,\ymin+0.5);
%label
\draw (\xmin+0.5,\ymin+0.4) -- (\xmin+0.5,\ymin+0.6);
\draw (\xmin+1.7,\ymin+0.5) node[below,scale=1] {$\frac{3n}{11}$};
\draw[red] (\xmin+1.7,\ymin+0.5) node[above,scale=1] {red};
\draw (\xmin+3,\ymin+0.4) -- (\xmin+3,\ymin+0.6);
\draw (\xmin+6.3,\ymin+0.5) node[below,scale=1] {$\frac{7n}{11}$};
\draw[blue] (\xmin+6.3,\ymin+0.5) node[above,scale=1] {blue};
\draw (\xmin+9.5,\ymin+0.4) -- (\xmin+9.5,\ymin+0.6);
\draw (\xmin+10.2,\ymin+0.5) node[below,scale=1] {$\frac{n}{11}$};
\draw[red] (\xmin+10.2,\ymin+0.5) node[above,scale=1] {red};
\draw (\xmin+11,\ymin+0.4) -- (\xmin+11,\ymin+0.6);
\end{tikzpicture}
\caption{The Optimal Coloring for $x+2y=z$}
\label{fig:optimalcoloring2}
\end{center}
\end{figure}

\noindent This coloring gives us $\frac{n^2}{44}+\mathcal{O}(n)$ monochromatic triples.
\vspace{0.2cm}

\noindent Next, we look for the lower bound of the minimum.  From Lemma \ref{M2} and Proposition \ref{N2}, we immediately get that
\begin{align*}
|\mathcal{M}(n)| &\geq \frac{n^2}{4}- \dfrac{1}{2}\left( \dfrac{\mu_R\mu_B}{2}+\mu_R\mu_{B_1}
+\mu_B\mu_{R_1}+|N_x^-|-|N_y^+| \right)+\mathcal{O}(n) \\
&\geq \frac{n^2}{4}-\frac{5n^2}{22}+\mathcal{O}(n)\\
&= \frac{n^2}{44}+\mathcal{O}(n).
\end{align*}

\noindent Because the lower and upper bounds match, we have therefore shown the desired result.

\end{proof}

%%%%%%%%%%%%%%%%%%%%%%%%%%%%%%%%%%%%%%%%%%%%%%%
\section{The General Case \texorpdfstring{$x+ay=z, \;\ a \geq 2$}{}}

We now generalize our result.

\begin{theorem} \label{limitbreak}
Over all 2-colorings of $[1,n]$, the minimum number of monochromatic triples satisfying $x+ay=z, \, a\geq 2$ is $\frac{n^2}{2a(a^2+2a+3)}+\mathcal{O}(n)$.
\end{theorem}

\noindent The set up of this proof is similar to the set up in Section \ref{a=2}. We will outline it here.

\begin{definition}
Denote by $\mu_{B_1}$ and $\mu_{R_1}$ the number of blue and red colorings respectively on $\left[1,\frac{n}{a}\right]$.  Furthermore, denote by $\mu_{B_2}$ and $\mu_{R_2}$ the number of blue and red colorings respectively on $\left(\frac{n}{a},n\right]$.
\end{definition}

\begin{definition}
The sets of non-monochromatic pairs in $[1,n]\times[1,\frac{n}{a}]$ 
are defined as follows:
\begin{equation*}
\begin{split}
N_x^-=\left\{(x,y)|\, x+ay\leq n,x>ay\right\}\\
N_x^+=\left\{(x,y)| \, x+ay>n,x>ay\right\}\\
N_y^-=\left\{(x,y)| \, x+ay\leq n,x<ay\right\}\\
N_y^+=\left\{(x,y)| \, x+ay>n,x<ay\right\}
\end{split}
\end{equation*}

\begin{figure}[ht]
\begin{center}
\begin{tikzpicture}[scale=0.65]
%grid
\def\xmin{-4.5}
\def\xmax{4.5}
\def\ymin{-4.5}
\def\ymax{4.5}
%axes
\draw[very thick, <->] (\xmin,\ymin+0.5)--(5.5,\ymin+0.5) node[right,scale=1]{$x$};
\draw[very thick, <->] (\xmin+0.5,\ymin)--(\xmin+0.5, 0.3) node[above,scale=1]{$y$};
%label
\draw (\xmax-0.5,\ymin+0.5) node[below,scale=1] {$n$};
\draw (\xmin+0.5,-1) node[left,scale=1] {$\frac{n}{a}$};
%graph
\draw[very thick] (\xmin+0.5,-1) -- (\xmax-0.5,-1);
\draw[very thick] (\xmin+0.5,\ymin+0.5) -- (\xmax-0.5,-1);
\draw[very thick] (\xmin+0.5,-1) -- (\xmax-0.5,\ymin+0.5);
\draw[very thick] (\xmax-0.5,\ymin+0.5) -- (\xmax-0.5,-1);
\draw (0,-1.8) node[scale=1]{$N_y^+$};
\draw (0,\ymin+1) node[scale=1]{$N_x^-$};
\draw (\xmin+2,\ymin+1.9) node[scale=1]{$N_y^-$};
\draw (\xmax-2,\ymin+1.9) node[scale=1]{$N_x^+$};
%dots
\draw[fill] (\xmin+2.5,\ymin+3.1) circle [radius=0.1];
\draw[fill] (\xmax-2.5,\ymin+3.1) circle [radius=0.1];
\draw[fill] (\xmin+2.5,\ymin+0.8) circle [radius=0.1];
\draw[fill] (\xmax-2.5,\ymin+0.8) circle [radius=0.1];
\end{tikzpicture}
\caption{The sets $N_x^-$, $N_x^+$ $N_y^-$ and $N_y^+$.}
\label{fig:graphxplusaywithdots}
\end{center}
\end{figure}
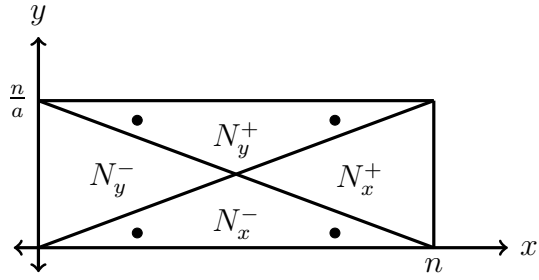

\end{definition}

\noindent It is now easy to adapt this notation to prove the following analog to Proposition \ref{nonmono2}.

\begin{proposition}\label{nonmonoa}
$|\mathcal{N}(n)| \leq \dfrac{1}{2}\left( \dfrac{\mu_R\mu_B}{a}+\mu_R\mu_{B_1}+\mu_B\mu_{R_1}
+|N_x^-|-|N_y^+| \right)+\mathcal{O}(n).$
\end{proposition}

\begin{definition}\label{1overapair}
Define by $S$ the set of pairs of the form  $\{s,\frac{n}{a}+1-s\}$ where $1 \leq s \leq \frac{n}{2a}.$
\end{definition}

\noindent The values of $D_a:=|N_x^-|-|N_y^+|$ can be bounded the same way as in the previous section.

\begin{proposition}\label{dabound}
Assume, without loss of generality, that $\mu_B \geq \mu_R$ and suppose the number of non-monochromatic pairs in S, $\gamma_1 n$, is fixed. Then
\[ D_a \leq   2 A_1,\]
where $A_1$ is the largest possible area that can be placed under the curve in Figure \ref{fig:dagraphcombined}, with a base of length $\gamma_1 n$ for $\gamma_1 \leq \frac{1}{2a}$.
\end{proposition}

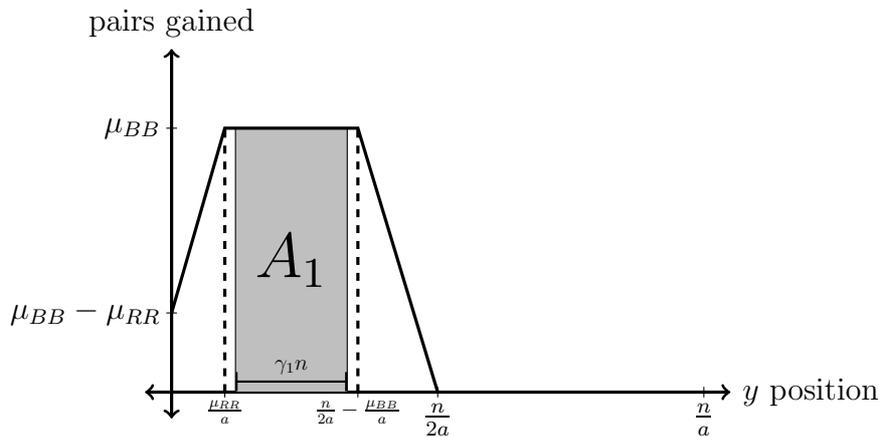
\begin{figure}[ht]
\begin{center}
\begin{tikzpicture}[scale=0.7]
%grid
\def\xmin{-5.5}
\def\xmax{5.5}
\def\ymin{-4.5}
\def\ymax{4.5}
%axes
\draw[very thick, <->] (\xmin,\ymin+2.5)--(\xmax,\ymin+2.5) node[right,scale=1]{$y$ position};
\draw[very thick, <->] (\xmin+0.5,\ymin+2)--(\xmin+0.5,\ymax) node[above,scale=1]{pairs gained};
%label
\draw (0,\ymin+2.5) node[below,scale=1] {$\frac{n}{2a}$};
\draw (0,\ymin+2.4) -- (0,\ymin+2.6);
\draw (\xmax-0.5,\ymin+2.5) node[below,scale=1] {$\frac{n}{a}$};
\draw (\xmax-0.5,\ymin+2.4) -- (\xmax-0.5,\ymin+2.6);
\draw (\xmin+0.5,\ymax-1.5) node[left,scale=1] {$\mu_{BB}$};
\draw (\xmin+0.4,\ymax-1.5) -- (\xmin+0.6,\ymax-1.5);
\draw (\xmin+0.5,\ymax-5) node[left,scale=1] {$\mu_{BB}-\mu_{RR}$};
\draw (\xmin+0.4,\ymax-5) -- (\xmin+0.6,\ymax-5);
\draw (\xmin+1.5,\ymin+2.5) node[below,scale=0.7] {$\frac{\mu_{RR}}{a}$};
\draw (\xmin+1.5,\ymin+2.4) -- (\xmin+1.5,\ymin+2.6);
\draw (\xmin+4,\ymin+2.5) node[below,scale=0.7] {$\frac{n}{2a}-\frac{\mu_{BB}}{a}$};
\draw (\xmin+4,\ymin+2.4) -- (\xmin+4,\ymin+2.6);
%shading
\draw[fill=lightgray] (\xmin+1.7,\ymin+2.5) rectangle (\xmin+3.8,\ymax-1.5);
\draw [|-|, thick] (\xmin+1.7,\ymin+2.7) -- (\xmin+3.8,\ymin+2.7);
\draw (\xmin+2.75,\ymin+2.7) node[above,scale=0.7] {$\gamma_1 n$};
\draw (\xmin+2.75,\ymin+5) node[scale=2]{$A_1$};
%graph
\draw[very thick] (\xmin+0.5,\ymax-5) -- (\xmin+1.5,\ymax-1.5) -- (\xmin+4,\ymax-1.5) -- (0,\ymin+2.5);
\draw[very thick, dashed] (\xmin+1.5,\ymin+2.5)--(\xmin+1.5,\ymax-1.5);
\draw[very thick, dashed] (\xmin+4,\ymin+2.5)--(\xmin+4,\ymax-1.5);
\end{tikzpicture}
\caption{The upper bound of $D_a$}
\label{fig:dagraphcombined}
\end{center}
\end{figure}

\noindent By combining Propositions \ref{nonmonoa} and \ref{dabound},
we obtain the upper bound for  $|\mathcal{N}(n)|$ as follows:

\begin{proposition}\label{Na}
Over all 2-colorings of $[1,n]$, the maximum number of non-monochro-matic triples satisfying $x+ay=z,\, a \geq 2$ is
\begin{equation*} 
|\mathcal{N}(n)| \leq  \frac{n^2}{2a}-\frac{n^2}{2a(a^2+2a+3)}+\mathcal{O}(n).\\
\end{equation*}
\end{proposition}

\begin{proof} Without loss of generality, we assume $\mu_B \geq \mu_R$. Suppose

\[ \Delta := \dfrac{1}{2}\left( \dfrac{\mu_R\mu_B}{a}+\mu_R\mu_{B_1}
+\mu_B\mu_{R_1}+2A_1 \right).\]

\noindent Propositions \ref{nonmonoa} and \ref{dabound} show that
optimizing $\Delta$ will give us the upper bound for $|\mathcal{N}(n)|.$  We will call this optimum $\Delta_{max}$.  In order to find $\Delta_{max}$, we consider three different cases to compute $A_1$.  Like before, each case will be subjected to the conditions listed in Figure \ref{fig:global}. \begin{description}
\item[Case 1:] $\gamma_1 n <\frac{\gamma n}{a}$.
\item[Case 2:] $\frac{\gamma n}{a}\leq\gamma_1 n <\frac{\mu_R}{a}$.
\item[Case 3:]  $\frac{\mu_R}{a}\leq\gamma_1 n$.
\end{description}
Here, we show only the details for Case 3A which will give us the best upper bound like in the previous section.
\begin{align*}
A_1&=\mu_{BB}\cdot\gamma_1 n-\frac{a}{2}\cdot\left(\frac{\mu_{RR}}{a}\right)^2
-\frac{a}{2}\cdot\left(\gamma_1 n-\frac{\gamma n}{a}-\frac{\mu_{RR}}{a}\right)^2\\
&=\dfrac{a}{2}\left(\frac{n}{a}-\gamma_1 n\right)\cdot\gamma_1 n
-\frac{1}{2a}\left((\gamma n+\mu_{RR})^2+\mu_{RR}^2\right)\\
&=\frac{a}{2}\left(\frac{n}{a}-\gamma_1 n\right)\cdot\gamma_1 n
-\frac{(\gamma n)^2}{4a}-\frac{\mu_{R}^2}{4a}.\\
\end{align*}

\noindent Similar to before, we want $\gamma_1 n$ to be as close to $\frac{n}{2a}$ as possible and $\gamma n$ should be as small as possible.  This is achieved by setting  $\gamma_1 n = \mu_{B_1}$ and $\gamma n=\mu_{R_1}-\mu_{R_2}.$  Then
\[\Delta_{max} = \frac{n^2}{2a}-\frac{n^2}{2a(a^2+2a+3)}+\mathcal{O}(n),\]
which is attained when  $\mu_{R_1} = \frac{a+1}{a^2+2a+3}$ and $\mu_{R} = \frac{a+2}{a^2+2a+3}.$
\end{proof} 
\vspace{0.2cm}

\begin{proof}[Proof of Theorem \ref{limitbreak}]
An upper bound of the minimum can be obtained from a coloring on $[1,n]$.  We color $[R,B,R]$ with the ratio $\left[1,a+\frac{1}{a+1},\frac{1}{a+1}\right]$ as illustrated in Figure \ref{fig:optimalcoloringa}, which was discovered in \cite{butler} and \cite{aek}.
\vspace{0.2cm}

\begin{figure}[ht]
\begin{center}
\begin{tikzpicture}[scale=0.9]
%grid
\def\xmin{-0.5}
\def\xmax{11.5}
\def\ymin{-1}
\def\ymax{1}
%axes
\draw[very thick, <-] (\xmin,\ymin+0.5)--(\xmin+0.5,\ymin+0.5);
\draw[very thick, red, -] (\xmin+0.5,\ymin+0.5)--(\xmin+3,\ymin+0.5);
\draw[very thick, blue, -] (\xmin+3,\ymin+0.5)--(\xmin+9.5,\ymin+0.5);
\draw[very thick, red, -] (\xmin+9.5,\ymin+0.5)--(\xmin+11,\ymin+0.5);
\draw[very thick, ->] (\xmin+11,\ymin+0.5)--(\xmin+11.5,\ymin+0.5);
%label
\draw (\xmin+0.5,\ymin+0.4) -- (\xmin+0.5,\ymin+0.6);
\draw (\xmin+1.7,\ymin+0.5) node[below,scale=1] {1};
\draw[red] (\xmin+1.7,\ymin+0.5) node[above,scale=1] {red};
\draw (\xmin+3,\ymin+0.3) node[below,scale=1] {:};
\draw (\xmin+3,\ymin+0.4) -- (\xmin+3,\ymin+0.6);
\draw (\xmin+6.3,\ymin+0.5) node[below,scale=1] {$a+\frac{1}{a+1}$};
\draw[blue] (\xmin+6.3,\ymin+0.5) node[above,scale=1] {blue};
\draw (\xmin+9.5,\ymin+0.3) node[below,scale=1] {:};
\draw (\xmin+9.5,\ymin+0.4) -- (\xmin+9.5,\ymin+0.6);
\draw (\xmin+10.2,\ymin+0.5) node[below,scale=1] {$\frac{1}{a+1}$};
\draw[red] (\xmin+10.2,\ymin+0.5) node[above,scale=1] {red};
\draw (\xmin+11,\ymin+0.4) -- (\xmin+11,\ymin+0.6);
\end{tikzpicture}
\caption{The Optimal Coloring for $x+ay=z$}
\label{fig:optimalcoloringa}
\end{center}
\end{figure}
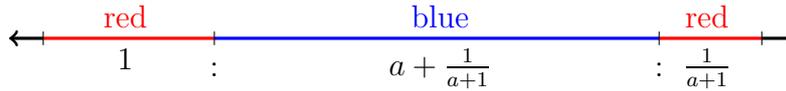

\noindent This coloring gives us $\frac{n^2}{2a(a^2+2a+3)}+\mathcal{O}(n)$  monochromatic triples, $(x,y,x+ay)$.
\vspace{0.2cm}

\noindent For the lower bound of the minimum, we use Proposition \ref{Na} to get
\begin{align*}
|\mathcal{M}(n)| &= \frac{n^2}{2a}- |\mathcal{N}(n)| \\
&\geq \frac{n^2}{2a(a^2+2a+3)}+\mathcal{O}(n).
\end{align*}

\noindent Because the lower and upper bounds match, we have therefore shown the desired result.
\end{proof}

%%%%%%%%%%%%%%%%%%%%%%%%%%%%%%%%%%%%%%%%%%%%%%%
\section{Conjectures}

In this section, we present conjectures on variations of Graham's original problem.  Denote by $R,B,G$ the colors red, blue, and green respectively.
\vspace{0.2cm}

\begin{enumerate}
\item{Equation: $ax+by=az$ where $a,b$ are integers.
\begin{enumerate}
\item{Case 1: $a>b\geq2, \;\ gcd(a,b)=1$\\
\noindent The coloring that gives the minimum number of monochromatic solutions over any 2-coloring of $[1,n]$ is
\[\left[(R^{a-1},B)^{\frac{n}{a}} \right]. \]}
\item{Case 2: $b>a\geq2, \;\ gcd(a,b)=1$\\
\noindent The coloring that gives the minimum number of monochromatic solutions over any 2-coloring of $[1,n]$ is
\[\left[(R^{a-1},B)^{\frac{n}{b}},\;\ R^{(\frac{b-a}{b})n}\right].\]}
\end{enumerate}
\noindent This has also been conjectured in \cite[p. 409]{butler}.
}
\item{Equation: $x+y+w=z$\\
\noindent The coloring that gives the minimum number of monochromatic solutions over any 2-coloring of $[1,n]$ is 
\[\left[R^{\frac{3(10-\sqrt{3})n}{97}},B^{\frac{(6+\sqrt{3})(10-\sqrt{3})n}{97}},R^{\frac{(10-\sqrt{3})n}{97}}\right],\]
\noindent with the number of monochromatic solutions to be 
\[\dfrac{n^3}{12(10+\sqrt{3})^2}+\mathcal{O}(n^2).\]}
\item{Equation: $x+y=z$ \\
\noindent The coloring that gives the maximum number of rainbow solutions over any 3-coloring of $[1,n]$ is 
\[\left[(R,B)^{\frac{n}{5}},(G,B)^{\frac{3n}{10}}\right],\]
\noindent with the number of rainbow solutions to be
\[\dfrac{n(n+1)}{10}.\]}
\end{enumerate}

%%%%%%%%%%%%%%%%%%%%%%%%%%%%%%%%%%%%%%%%%%%%%%


\begin{thebibliography}{2}

\bibitem{butler}
Steve Butler, Kevin P. Costello, and Ron Graham, \textit{Finding Patterns Avoiding Many Monochromatic Constellations}, Experimental Mathematics, \textbf{19:4} (2010), 399-411

\bibitem{datmono} 
Boris Datskovsky,
\textit{On the number of monochromatic Schur triples}, Advances in Applied Math, \textbf{31} (2003), 193-198.

\bibitem{erdos}
Paul Erdős and George Szekeres. \textit{A combinatorial problem in geometry}, Compositio Mathematica, \textbf{2} (1935), 463–470.

\bibitem{goodman}
Adolph Goodman, \textit{On Sets of Acquaintances and Strangers at any Party}, The American Mathematical Monthly, Vol. 66, No. 9 (1959), 778-783.

\bibitem{graham}
Ronald Graham, Vojtech R\"odl, and Andrzej Ruci\'nski, \textit{On Schur Properties of Random Subsets of Integers}, Journal of Number Theory, \textbf{61} (1996), 388-408.

\bibitem{robramsey} 
Bruce Landman and Aaron Robertson,
\textit{Ramsey Theory on the Integers},
AMS, 2004

\bibitem{zeilberger}
Aaron Robertson and Doron Zeilberger, \textit{A 2-coloring of $[1,n]$ can have $\frac{n^2}{22}+\mathcal{O}(n)$ monochromatic Schur triples, but not less!}, Electronic Journal of Combinatorics, \textbf{5} (1998), R19.

\bibitem{schoen}
Thomas Schoen, \textit{The number of monochromatic Schur Triples}, European Journal of Combinatorics, \textbf{20} (1999), 855-866.
 
\bibitem{aek}
Thotsaporn Thanatipanonda, \textit{On the Monochromatic Schur Triples Type Problem}, Electronic Journal of Combinatorics, \textbf{16} (2009), Research Paper 14, 11pp.

\bibitem{vanderwaarden}
Bartel Leedert van der Waerden, \textit{Beweis einer Baudetschen Vermutung}, Nieuw. Arch. Wisk. (in German) \textbf{15} (1927), 212–216.

\end{thebibliography}
\end{document}